\documentclass[a4paper, reqno, 14pt]{amsart}

\usepackage[usenames,dvipsnames]{color}
\usepackage{amsthm,amsfonts,amssymb,amsmath,amsxtra}
\usepackage{tikz}
\usepackage{mathabx}

\usepackage{tikz-cd}
\usepackage{todonotes,cancel}

\usepackage[all]{xy}
\SelectTips{cm}{}
\usepackage{xr-hyper}
\usepackage[colorlinks=true, citecolor=blue]{hyperref}
\usepackage{verbatim}

\usepackage[margin=1.25in]{geometry}
\usepackage{mathrsfs}

\RequirePackage{xspace}
\RequirePackage{etoolbox}
\RequirePackage{varwidth}
\RequirePackage{enumitem}
\RequirePackage{tensor}
\RequirePackage{mathtools}
\RequirePackage{longtable}
\RequirePackage{multirow}

\def\ge{\geqslant}
\def\le{\leqslant}
\def\a{\alpha}
\def\b{\beta}

\def\D{\Delta}

\def\s{\sigma}

\def\l{\lambda}

\def\i{^{-1}}

\def\<{\langle}
\def\>{\rangle}

\newcommand{\club}{{\clubsuit}}

\newcommand{\spade}{{\spadesuit}}

\newcommand{\BC}{\ensuremath{\mathbb {C}}\xspace}

\newcommand{{\BG}}{\ensuremath{\mathbb {G}}\xspace}

\newcommand{{\BK}}{\ensuremath{\mathbb {K}}\xspace}

\newcommand{\BR}{\ensuremath{\mathbb {R}}\xspace}

\newcommand{\BZ}{\ensuremath{\mathbb {Z}}\xspace}

\newcommand{\CB}{\ensuremath{\mathcal {B}}\xspace}

\newcommand{\CD}{\ensuremath{\mathcal {D}}\xspace}

\newcommand{\CP}{\ensuremath{\mathcal {P}}\xspace}

\newcommand{\id}{\ensuremath{\mathrm{id}}\xspace}

\newcommand{\Rp}{\mathbb{R}_{>0}}

\newcommand{\leJ}{ {\, {}^J \!\! \le \,} }
\newcommand{\lessJ}{ {\, {}^J \!\! < \,}}
 \newcommand{\preceqJ}{{}^J\!\! \!\preceq}

\newcommand{\Jell}{{}^{J}\ell}

\newcommand{\JCB}{{}^J\!\CB}
\newcommand{\OJCB}{{}^J\!\mathring{\CB}}
\newcommand{\OCB}{\mathring{\CB}}
\newcommand{\JHQ}{{}^J\!\hat{Q}}
\newcommand{\JQ}{{}^J\!{Q}}
\newcommand{\JB}{{}^J \! B}
\newcommand{\JU}{{}^J  U}
\newcommand{\Jc}{{}^J\!{c}}
\newcommand{\Jsigma}{{}^{J}\!\sigma}

\def\tW{\tilde W}

\newtheorem{thm}{Theorem}
\newtheorem*{theorema}{Theorem A}
\newtheorem*{theoremb}{Theorem B}
\newtheorem*{theoremc}{Theorem C}
\newtheorem*{theoremd}{Theorem D}
\newtheorem{prop}[thm]{Proposition}

\newtheorem{lem}[thm]{Lemma}

\newtheorem{cor}[thm]{Corollary}

\theoremstyle{definition}

\newtheorem{defi}[thm]{Definition}

\newtheorem{remark}[thm]{Remark}

\numberwithin{equation}{section}
\numberwithin{thm}{section}

\setitemize[0]{leftmargin=*,itemsep=\the\smallskipamount}
\setenumerate[0]{leftmargin=*,itemsep=\the\smallskipamount}

\renewcommand{\to}{%
   \ifbool{@display}{\longrightarrow}{\rightarrow}%
   }
\let\shortmapsto\mapsto
\renewcommand{\mapsto}{%
   \ifbool{@display}{\longmapsto}{\shortmapsto}%
   }
\newlength{\olen}
\newlength{\ulen}
\newlength{\xlen}
\newcommand{\xra}[2][]{%
   \ifbool{@display}%
      {\settowidth{\olen}{$\overset{#2}{\longrightarrow}$}%
       \settowidth{\ulen}{$\underset{#1}{\longrightarrow}$}%
       \settowidth{\xlen}{$\xrightarrow[#1]{#2}$}%
       \ifdimgreater{\olen}{\xlen}%
          {\underset{#1}{\overset{#2}{\longrightarrow}}}%
          {\ifdimgreater{\ulen}{\xlen}%
             {\underset{#1}{\overset{#2}{\longrightarrow}}}
             {\xrightarrow[#1]{#2}}}}%
      {\xrightarrow[#1]{#2}}
   }
\makeatother
\newcommand{\xyra}[2][]{%
   \settowidth{\xlen}{$\xrightarrow[#1]{#2}$}%
   \ifbool{@display}%
      {\settowidth{\olen}{$\overset{#2}{\longrightarrow}$}%
       \settowidth{\ulen}{$\underset{#1}{\longrightarrow}$}%
       \ifdimgreater{\olen}{\xlen}%
          {\mathrel{\xymatrix@M=.12ex@C=3.2ex{\ar[r]^-{#2}_-{#1} &}}}%
          {\ifdimgreater{\ulen}{\xlen}%
             {\mathrel{\xymatrix@M=.12ex@C=3.2ex{\ar[r]^-{#2}_-{#1} &}}}
             {\mathrel{\xymatrix@M=.12ex@C=\the\xlen{\ar[r]^-{#2}_-{#1} &}}}}}%
      {\mathrel{\xymatrix@M=.12ex@C=\the\xlen{\ar[r]^-{#2}_-{#1} &}}}%
   }
\makeatletter
\newcommand{\xla}[2][]{%
   \ifbool{@display}%
      {\settowidth{\olen}{$\overset{#2}{\longleftarrow}$}%
       \settowidth{\ulen}{$\underset{#1}{\longleftarrow}$}%
       \settowidth{\xlen}{$\xleftarrow[#1]{#2}$}%
       \ifdimgreater{\olen}{\xlen}%
          {\underset{#1}{\overset{#2}{\longleftarrow}}}%
          {\ifdimgreater{\ulen}{\xlen}%
             {\underset{#1}{\overset{#2}{\longleftarrow}}}
             {\xleftarrow[#1]{#2}}}}%
      {\xleftarrow[#1]{#2}}
   }
\newcommand{\isoarrow}{%
   \ifbool{@display}{\overset{\sim}{\longrightarrow}}{\xrightarrow\sim}%
   }

\begin{document}

\title[]{Product structure and regularity theorem for totally nonnegative flag varieties}
\author[Huanchen Bao]{Huanchen Bao}
\address{Department of Mathematics, National University of Singapore, Singapore.}
\email{huanchen@nus.edu.sg}

\author[Xuhua He]{Xuhua He}
\address{The Institute of Mathematical Sciences and Department of Mathematics, The Chinese University of Hong Kong, Shatin, N.T., Hong Kong SAR, China}
\email{xuhuahe@math.cuhk.edu.hk}
\thanks{}


\begin{abstract}
The totally nonnegative flag variety was introduced by Lusztig. It has enriched combinatorial, geometric, and Lie-theoretic structures. In this paper, we introduce a (new) $J$-total positivity on the full flag variety of an arbitrary Kac-Moody group, generalizing the (ordinary) total positivity. 

We show that the $J$-totally nonnegative flag variety has a cellular decomposition into totally positive $J$-Richardson varieties. Moreover, each totally positive $J$-Richardson variety admits a favorable decomposition, called a product structure. Combined with the generalized Poincare conjecture, we prove that the closure of each totally positive $J$-Richardson variety is a regular CW complex homeomorphic to a closed ball. Moreover, the $J$-total positivity on the full flag provides a model for the (ordinary) totally nonnegative partial flag variety. As a consequence, we prove that the closure of each (ordinary) totally positive Richardson variety is a regular CW complex homeomorphic to a closed ball, confirming conjectures of Galashin, Karp and Lam in \cite{GKL}.
 
\end{abstract}

\maketitle

\tableofcontents

\section{Introduction}

\subsection{Totally nonnegative flag varieties of reductive groups}

The theory of total positivity on the reductive groups $G$ and their partial flag varieties $\CP_K$ was introduced by Lusztig in the seminal work \cite{Lu94}. The totally nonnegative partial flag variety $\CP_{K, \ge 0}$ is a ``remarkable polyhedral subspace'' (cf. \cite{Lu94}). It has many nice combinatorial, geometric, and Lie-theoretic properties. And it has been used in many other areas, such as cluster algebras \cite{FZ}, the Grassmann polytopes \cite{Lam16}, the physics of scattering amplitudes \cite{AHBC16}. 

We give a quick review of the definition and some nice properties of $\CP_{K,\ge 0}$. Let $G$ be a connected reductive group, split over $\BR$ and $B^\pm=T U^\pm$ be the Borel and opposite Borel subgroups of $G$. The full flag variety $\CB=G/B^+$ admits the decompositions into Schubert cells and opposite Schubert cells, both indexed by the Weyl group $W$ of $G$. The intersection of a Schubert cell $B^+ w B^+/B^+$ with an opposite Schubert cell $B^- v  B^+ / B^+$ is called an (open) Richardson variety, and is denoted by $\CB_{v, w}$. The variety $\CB_{v, w}$ is nonempty if and only if $v \le w$, where $\le$ is the Bruhat order on $W$. 

Let $I$ be the set of simple roots in $G$. Let $P^+_K \supset B^+$ be the standard parabolic subgroup associated to a subset $K$ of $I$. For the partial flag $\CP_K=G/P^+_K$, we have the decomposition into the projected Richardson varieties $\CP_K=\bigsqcup_{\a \in Q_K} \CP_{K, \a}$. The definition and the closure relation of the projected Richardson varieties are more complicated and we skip the details in the introduction. 

Let $U^-_{\ge 0}$ be the totally nonnegative part of $U^-$. 
The totally nonnegative part $\CP_{K, \ge 0}$ of the partial flag variety $\CP_K$ is by definition, the closure of $U^-_{\ge 0} P^+_K/ P^+_K$ in $\CP_K$. In the case where $\CP_K$ is the Grassmannian, $\CP_{K, \ge 0}$ is the totally nonnegative Grassmannian \cite{Pos}. 

The totally positive projected Richardson variety $\CP_{K, \a, >0}$ is, by definition, the intersection of the totally nonnegative partial flag $\CP_{K, \ge 0}$ with the projected Richardson variety $\CP_{K, \a}$. We then have the stratification $$\CP_{K, \ge 0}=\bigsqcup_{\a \in Q_K} \CP_{K, \a, >0}.$$ 

We have many remarkable properties on the totally positive projected Richardson varieties. 

\begin{enumerate}
    \item $\CP_{K, \ge 0}$ admits a natural monoid action of $G_{\ge 0}$ and a natural duality (see \cite{Lus-1}); 
        
    \item $\CP_{\ge 0}$ admits an representation-theoretic interpretation via canonical basis (see \cite{Lus-1} and \cite{Lu98});
    
    \item $\CP_{K, \a, >0}$ is a cell and is a connected component of $\CP_{K, \a}(\BR)$ (see \cite{Ri99});
    
    \item The closure of $\CP_{K, \a, >0}$ is a union of $\CP_{K, \a', >0}$ for some $\a'$ (see \cite{Ri06}); 
    
    \item The cell decomposition $\overline{\CP_{K, \a, >0}}=\bigsqcup_{\a'} \CP_{K, \a', >0}$ is a regular CW complex. 
\end{enumerate}

The last property is called the regularity theorem of $\CP_{K,\ge 0}$. In particular, the closure $\overline{\CP_{K, \a, >0}}$ is homeomorphic to a closed ball. It was conjectured by Postnikov for totally nonnegative Grassmannian and by Williams \cite{W07} for totally nonnegative partial flag varieties of split real reductive groups. Important progress has been made in \cite{PSW09}, \cite{RW}, \cite{RW10}, \cite{GKL17}, \cite{GKL18}. It was finally established by Galashin, Karp and Lam \cite{GKL}.

\subsection{Totally nonnegative Kac-Moody flag varieties}
The theory of total positivity on the reductive groups and their flag varieties have been generalized to arbitrary Kac-Moody groups by Lusztig in a series of papers \cite{Lu-2}, \cite{Lu-positive}, \cite{Lu-flag}, \cite{Lu-par} and \cite{Lu-Spr}, and by us in \cite{BH20}. For the full flag variety of an arbitrary Kac-Moody group, we proved in \cite{BH20} that the totally nonnegative flag variety $\CB_{\ge 0}$ has a representation-theoretic interpretation, is a union of totally positive Richardson varieties, and each totally positive Richardson variety is a cell. 

However, the closure relations among the cells and the geometric structure of these closures were not established. For reductive groups, there is a natural duality coming from $B^+ \leftrightarrow B^-$, which plays a significant role in establishing geometric properties of the flag varieties. Such duality does not exist for Kac-Moody groups, which leads to extra difficulty in the study of totally nonnegative flag varieties for the general Kac-Moody groups than the reductive groups.  We shall overcome the obstacles and establish results in the general setting of $J$-total positivity using the ``product structure".

\subsection{$J$-total positivity} Unless otherwise stated, in the rest of this paper we assume that $G$ is a Kac-Moody group, split over $\BR$. We fix a subset $J$ of $I$. Let $\JB^+ \subset P^+_J$ be the Borel subgroup opposite to $B^+$ and $\JB^- \subset P^-_J$ be the Borel subgroup opposite to $B^-$. The $\JB^+$-orbits on $\CB=G/B^+$ are called the $J$-Schubert cells and $\JB^-$-orbits on $\CB=G/B^+$ are called the opposite $J$-Schubert cells, respectively. 
For $v, w \in W$, the open $J$-Richardson variety is defined to be 
$$
\JCB_{v, w}=\JB^+ w B^+/B^+ \bigcap \JB^- v B^+/B^+.
$$ 
It is known that $\JCB_{v, w} \neq \emptyset$ if and only if $v \leJ w$, where $\leJ$ is the $J$-twisted Bruhat order. 

Our motivation to study the $J$-Richardson varieties comes from the partial flag varieties. The projected Richardson varieties in a partial flag variety $\CP_K$ and their geometric structures come from the projection map $\CB \to \CP_K$. Roughly speaking, the projection map $\CB \to \CP_K$ folds the Richardson varieties in a rather complicated way, which makes the projected Richardson varieties  rather complicated to study. In \cite{BH21} we introduced an ``atlas model'' $\CP_K \dashrightarrow \tilde \CB$ of the partial flag variety and regarded the projected Richardson varieties in $\CP_K$ as certain $J$-Richardson varieties in the full flag variety $\tilde \CB$ of another (larger) Kac-Moody group.  
 
However, the ``atlas model'' $\CP_K \dashrightarrow \tilde \CB$ is not compatible with the total positivity on $\CP_K$ and the ordinary total positivity $\tilde \CB$. This should not be a surprise, as the ordinary total positivity on $\tilde \CB$ is suitable for the decomposition into the (ordinary) Richardson varieties, not the $J$-Richardson varieties. To provide a ``model'' for the total positivity on $\CP_K$, we introduce the $J$-total positivity.  The $J$-total positivity on the flag variety is ``compatible" with the stratification by $J$-Richardson varieties.  

It is worth mentioning that when $J =\emptyset$, the $J$-Schubert (resp. opposite $J$-Schubert) varieties are just the Schubert (resp. opposite Schubert) varieties;  the $J$-total positivity coincides with the ordinary total positivity. Therefore, our main results apply to the setting of the ordinary total positivity. If the Weyl group $W_J$ is finite, then the $J$-total positivity can be obtained from the ordinary total positivity by multiplying $\dot w_J$ on the left, where $w_J$ is the longest element of $W_J$. In general, the $J$-total positivity is quite different from the ordinary total positivity.

\subsection{The main results on the $J$-total positivity} We set $${}^J U^-_{\ge 0}=\{h_1 \pi_J(h_2) \i h_2; h_1 \in U^-_{J, \ge 0}, h_2 \in U^-_{\le 0}\}.$$ Here $U^-_J$ is the unipotent radical of the opposite Borel subgroup in the Levi subgroup $L_J$ of $G$ and $\pi_J$ is the projection map from the opposite parabolic subgroup $P^-_J$ to its Levi subgroup $L_J$. We define the $J$-totally nonnegative flag variety $$\JCB_{\ge 0}=\overline{{}^J U^-_{\ge 0} \cdot B^+} \subset \CB.$$ 

For any $w_1 \leJ w_2$, we set $\JCB_{w_1, w_2, >0}=\JCB_{\ge 0} \bigcap \JCB_{w_1, w_2}$. We call $\JCB_{w_1, w_2, >0}$ the totally positive $J$-Richardson variety\footnote{This should be called $J$-totally positive $J$-Richardson variety to be precise. But since we never consider the interaction between the ordinary total positivity and $J$-Richardson varieties, this should  not cause any confusion.}. Note that the definition of ${}^J U^-_{\ge 0}$ is a mixture of the totally positive and totally negative parts $U^-_{\le 0}$ of $U^-$. The $J$-total positivity is more difficult to study than the ordinary total positivity on $\CB$. Some major differences between the $J$-total positivity and the ordinary total positivity are 

\begin{itemize}
    \item the totally nonnegative flag $\CB_{\ge 0}$ admits a natural action of the totally nonnegative monoid $G_{\ge 0}$, while the $J$-totally nonnegative flag $\JCB_{\ge 0}$ only admits a natural action of totally nonnegative submonoid $L_{J, \ge 0}$; 
    
    \item the totally nonnegative flag $\CB_{\ge 0}$ has a nice representation-theoretic interpretation via Lusztig's canonical basis. In contrast, the positivity property of the canonical basis is not preserved for the $J$-total positivity. 
\end{itemize}

It is worth mentioning that the symmetry (of $G_{\ge 0}$), the representation-theoretic interpretation and the duality $B^+ \leftrightarrow B^-$ play a crucial role in the previous study of the totally nonnegative flag $\CB_{\ge 0}$ of reductive groups. However, none of these features are available for the $J$-totally nonnegative flag variety of a general Kac-Moody group. Thus we need to develop a new strategy to study the $J$-total positivity. 

Our starting point is the open covering $\CB=\bigcup_{w \in W} w U^- B^+/B^+$ and the isomorphisms 
$$
\Jc_w: w U^- B^+/B^+ \cong \JB^+ w B^+/B^+ \times \JB^- w B^+/B^+.
$$ 
The idea of such isomorphism dates back to Kazhdan and Lusztig \cite{KL}, see also \cite{KWY}. 

Our first main result on $J$-total positivity is the following.  Part (2)  is new even for the ordinary total positivity for the full flag variety of reductive groups. 

\begin{theorema}[Proposition~\ref{prop:J}, Theorem~\ref{thm:J}]\label{thmA}
Let $w_1 \leJ w_3 \leJ w_2$. Then 

(1) $\JCB_{w_1, w_2, >0} \subset w_3 U^- B^+/B^+$. 

(2) The map $\iota_{w_3}$ induces an isomorphism $$\JCB_{w_1, w_2, >0} \cong \JCB_{w_1, w_3, >0} \times \JCB_{w_3, w_2, >0}.$$
\end{theorema}

We call the isomorphism in part (2) of Theorem A {\it the product structure} of $\JCB_{w_1, w_2, >0}$. If we fix $w_3$, but let $w_1$ and $w_2$ vary, then we obtain an isomorphism 
\begin{equation}\label{eq:star}
\tag{$\star$} \bigsqcup_{w_1 \leJ w_3 \leJ w_2} \JCB_{w_1, w_2, >0} \cong \bigsqcup_{w_1 \leJ w_3} \JCB_{w_1, w_3, >0} \times \bigsqcup_{w_3 \leJ w_2} \JCB_{w_3, w_2, >0}.
\end{equation}

We call it the {\it product structure} of $\bigsqcup_{w_1 \leJ w_3 \leJ w_2} \JCB_{w_1, w_2, >0}$. It allows us to understand $\JCB_{w_1, w_2, >0}$ and its closure inductively. We can translate a geometric/topological question of calculating the closure to an algebraic question of calculating the image under the map ${}^J c_w$. As consequences of the product structure \eqref{eq:star} we obtain  

\begin{theoremb} [Theorem~\ref{thm:J}]\label{thmB}
Let $w_1 \leJ w_2$. Then 

(1) $\JCB_{w_1, w_2, >0}$ is a cell and is a connected component of $\JCB_{w_1, w_2}(\BR)$. 

(2) The closure of $\JCB_{w_1, w_2, >0}$ is $\bigsqcup_{w_1 \leJ w'_1 \leJ w'_2 \leJ w_2} \JCB_{w'_1, w'_2, >0}$. 
\end{theoremb}

Taking $J = \emptyset$, Theorem B generalizes the works of Rietsch in \cite{Ri99, Ri06} from reductive groups to Kac-Moody groups.

We also prove that 

\begin{theoremc}[Proposition~\ref{prop:compatible}]\label{thmC}
The Birkhoff-Bruhat atlas of \cite{BH21} sends a totally positive cell $\CP_{K, \a, >0}$ isomorphically to a totally positive cell $J$-Richardson variety ${}^J \! \tilde \CB_{w_1, w_2, >0}$ for certain $J$ and $w_1, w_2 \in \tW$. 
\end{theoremc}

Note that the partial flag variety $\CP_{K, \ge 0}$ does not have an obvious product structure as in Theorem A. This is reflected combinatorially on the lack of symmetry on the face poset of $\CP_{K, \ge 0}$. Theorem C allows us to study (inductively) the complicated totally positive projected Richardson varieties on $\CP_{K}$ using the product structure coming from totally positive $J$-Richardson varieties on $\tilde{\CB}$. This is a key ingredient in our proof of the regularity theorem for the totally positive projected Richardson varieties on $\CP_{K}$, which we will discuss in the next subsection. 

The $J$-total positivity for the full flag variety of any Kac-Moody group will also be applied to the study of the total positivity in many other spaces, such as the double flag varieties, the Bott-Samelson varieties, the double Bruhat cells and the wonderful compactifications. This will be done in future works. 

\subsection{Regularity Theorem} We establish the regularity theorem for the links of the identity in the totally positive cells in $U^-$, the (ordinary) totally positive cells in the partial flag varieties, and the totally positive $J$-Richardson varieties in the full flag varieties. 

\begin{theoremd}[Theorem~\ref{thm:CB}, Theorem~\ref{thm:CPK}, Theorem~\ref{thm:J}, Theorem~\ref{thm:lkregular}]\label{thmD}
All the following three spaces are regular CW complexes homeomorphic to closed balls: 
\begin{enumerate}
	\item  the link of the identity in $\overline{U^-_{w, >0}}$, for any $w\in W$;
	\item   the totally nonnegative projected Richardson variety $\overline{\CP_{K, \a, >0}}$;
	\item the totally nonnegative $J$-Richardson variety $\overline{\JCB_{w_1, w_2, >0}}$.
\end{enumerate}
\end{theoremd}

For reductive groups, the regularity of the link was first established by Hersh in \cite{Her};  the regularity of $\overline{\CP_{K, \a, >0}}$ was established by Galashin, Karp and Lam in \cite{GKL}. The generalization of regularity theorems in \cite{Her, GKL} for Kac-Moody groups was conjectured by Galashin, Karp and Lam in \cite[conjecture 10.2]{GKL}. Theorem D (1) \& (2) proves the conjectures, and part (3) is a new regularity result.

To prove regularity theorems, we follow \cite{GKL} for the use of the generalized Poincar\'e conjecture \cite{Sm61}, \cite{Fr82} and \cite{Pe02} as well as some general results on the poset topology. One then needs to show that each space $\overline{Y}$ above is a topological manifold with boundary $\overline{Y}-Y$. In the case where $Y=\CP_{K, \a, >0}$ for a reductive group, \cite{GKL} proved this result by constructing the Fomin-Shapiro atlas. Such construction relies on the affine  model \cite[\S 7]{GKL} and a detailed study of the admissible functions \cite[\S 5 \& \S 6]{GKL}. Another crucial fact is that the maps involved in the construction are the restrictions of smooth maps. Such construction of the Fomin-Shapiro atlas does not work for the ordinary total positivity for the Kac-Moody groups of infinite types nor the $J$-total positivity. 

Instead, we use the product structure to study (inductively) the open neighborhood of a point in each space $\overline{Y}$ in Theorem D. For the $J$-Richardson varieties, the product structure \eqref{eq:star} is established in Theorem A. The product structure of the links in the totally positive cells in $U^-$ is inherited from the product structure of the (ordinary) total positivity on the full flag variety, which is also established in Theorem A. 

As to the partial flag varieties, we do not have a obvious product structure. However, by Theorem C,  the ``atlas model'' $\CP_K \dashrightarrow \tilde \CB$ in \cite{BH20} translates the local structure of the totally nonnegative projected Richardson varieties in $\CP_K$ homeomorphically to the local structure of a $J$-Richardson variety in $\tilde \CB$. And thus we may use the product structure of the $J$-total positivity to understand the local structures of the totally positive projected Richardson varieties and establish the desired regularity theorem.

\vspace{.2cm}
\noindent {\bf Acknowledgement: } HB is supported by MOE grants R-146-000-294-133 and R-146-001-294-133. XH is partially supported by a start-up grant and by funds connected with Choh-Ming Chair at CUHK, and by Hong Kong RGC grant 14300221.


\section{Preliminaries}

Throughout this paper, unless stated otherwise, for any ind-scheme $X$ over $\BR$, we shall simply denote by $X$ its set of $\BC$-valued points, and denote by $X(\BR)$ its set of $\BR$-valued point. For any topological subspace $Y$ of $X(\BR)$, we denote by $\overline{Y}$ the closure with respect to the Hausdorff topology. 
\subsection{Minimal Kac-Moody groups}

Let $I$ be a finite set and $A=(a_{ij})_{i, j \in I}$ be a symmetrizable generalized Cartan matrix in the sense of \cite[\S 1.1]{Kac}. A {\it Kac-Moody root datum} associated to $A$ is a quintuple $\CD=(I, A, X, Y, (\a_i)_{i \in I}, (\a^\vee_i)_{i \in I})$, where $X$ is a free $\BZ$-module of finite rank with $\BZ$-dual $Y$, such that the elements $\a_i$ of $X$ and $\a^\vee_i$ of $Y$ satisfying $\<\a^\vee_j, \a_i\>=a_{ij}$ for $i, j \in I$. The {\it split minimal Kac-Moody group} $G$ over $\BR$ associated to the Kac-Moody root datum $\CD$ is the split group over $\BR$ generated by the split torus $T$ associated to $Y$ and the root subgroup $U_{\pm \a_i}$ for $i \in I$, subject to the Tits relations \cite{Ti87}. Let $U^+ \subset G$ (resp. $U^- \subset G$) be the subgroup generated by $U_{\a_i}$ (resp. $U_{-\a_i}$) for $i \in I$. Let $B^{\pm} \subset G$ be the Borel subgroup generated by $T$ and $U^{\pm}$, respectively.

For any $K \subset I$, let $L_K$ be the subgroup of $G$ generated by $T$ and $U_{\pm{\a_i}}$ for $i \in K$. Let $P^+_K$ be the standard parabolic subgroup of $G$ generated by $B^+$ and $U_{-\a_i}$ for $i \in K$ and $P^-_K$ be the opposite parabolic subgroup of $G$ generated by $B^-$ and $U_{\a_i}$ for $i \in K$. Let $U_{P^\pm_K}$ be the unipotent radical of $P^\pm_K$. We have the Levi decomposition $P^\pm_K=L_K \ltimes U_{P^\pm_K}$. 

For each $i \in I$, we fix isomorphisms $x_i: \BR \to U_{\a_i}$ and $y_i: \BR \to U_{-\a_i}$ such that the map
\[
\begin{pmatrix} 1 & a  \\ 0 & 1 \end{pmatrix} \mapsto x_i(a), \begin{pmatrix} b & 0  \\ 0 & b \i \end{pmatrix} \mapsto \a^\vee_i(a), \begin{pmatrix} 1 & 0  \\ c & 1 \end{pmatrix} \mapsto y_i(c)
\] defines a homomorphism $SL_2 \to G$. 

\subsection{Weyl groups} 
Let $W$ be the Weyl group of $G$. It is a Coxeter group with the set of generators $\{s_i\}_{i \in I}$. Let $\ell$ be the length function on $W$ and $\le$ be the Bruhat order on $W$. We have natural actions of $W$ on both $X$ and $Y$. 

For any $J \subset I$, let $W_J$ be the subgroup of $W$ generated by $s_j$ for $j \in J$. This is the Weyl group of the parabolic subgroup $P^+_J$ of $G$. Let $W^J$ be the set of minimal-length coset representatives of $W/W_J$ and ${}^J W$ be the set of minimal-length coset representatives of $W_J \backslash W$. The multiplication gives a natural bijection $W_J \times {}^J W \cong W$. For any $w \in W$, let $w_J$ be the unique element in $W_J$ and ${}^J w$ be the unique element in ${}^J W$ with $w=w_J \, {}^J w$. 

For any $i \in I$, we set $\dot s_i=x_i(1) y_i(-1) x_i(1) \in G(\BR)$. Let $w \in W$. By \cite[Proposition 7.57]{Mar}, for any reduced expression $w=s_{i_1} s_{i_2} \cdots s_{i_n}$ of $w$, the element $\dot s_{i_1} \dot s_{i_2} \cdots \dot s_{i_n}$ of $G$ is independent of the choice of the reduced expression. We denote this element by $\dot w$.

\subsection{The flag varieties} 

Let $\CB=G/B^+$ be the thin (full) flag variety of $G$ (see \cite{Kum}). For any $w \in W$, we set  $\mathring{\CB}_w=B^+ \dot{w} B^+ / B^+$ and $\mathring{\CB}^w=B^- \dot{w} B^+ / B^+$. We denote by ${\CB}_w$ and ${\CB}^w$ the Zariski closure of $\mathring{\CB}_w$ and $\mathring{\CB}^w$ in $\CB$, respectively. For $v, w \in W$, let $\mathring{\CB}_{v, w}=\mathring{\CB}_w \bigcap \mathring{\CB}^v$. This is an (open) Richardson variety of $\CB$. It is known that $\mathring{\CB}_{v, w} \neq \emptyset$ if and only if $v \le w$. In this case, $\dim \mathring{\CB}_{v, w}=\ell(w)-\ell(v)$. We have the decomposition $\CB=\bigsqcup_{v \le w} \mathring{\CB}_{v, w}$. Moreover, for any $v \le w$, the Zariski closure of ${\mathring{\CB}_{v, w}}$ is $\CB_w \bigcap \CB^v=\bigsqcup_{v \le v' \le w' \le w} \mathring{\CB}_{v', w'}$. 

 
\subsection{Regular CW complexes}
Let $X$ be a Hausdorff space. We call a finite disjoint union $X = \bigsqcup_{\alpha \in Q}X_{\alpha}$ a {\it regular CW complex} if it satisfies the following two properties.

\begin{enumerate}
\item For each $\alpha \in Q$, there exists a homeomorphism from a closed ball to $\overline{X}_\alpha$ mapping the interior of the ball to $X_\alpha$. 
\item For each $\alpha$, there exists $Q' \subset Q$, such that $\overline{X}_\alpha = \bigsqcup_{\beta \in Q'} X_\beta$.
 \end{enumerate}

The face poset of $X$ is the poset $(Q, \le)$, where $\beta \le \alpha$ if and only if $X_\beta \subset \overline{X}_\alpha$. 

We refer to \cite[\S 4]{BH21} and the references therein for the definitions of graded, thin, and shellable posets . We have the following result (see \cite{Bj2}).

\begin{thm}\label{thm:CW} 
Suppose that $X$ is a regular CW complex with face poset $Q$. If $Q \bigsqcup \{\hat{0}, \hat{1}\}$ (adjoining a minimum $\hat{0}$ and a maximum $\hat{1}$) is graded, thin, and shellable, then $X$ is homeomorphic to a sphere of dimension rank$(Q)-1$.
\end{thm}
 
 

 \subsection{The Poincare conjecture}
 
 Recall that an $n$-dimension topological manifold with boundary is a Hausdorff space $X$ such that every point $x \in X$ has an open neighborhood homeomorphic to either $\BR^n$, or $\BR_{\ge 0} \times \BR^{n-1}$ mapping $x$ to a point in $\{0\} \times  \BR^{n-1}$. In the latter case, we say that $x$ is on $\partial X$, the boundary of $X$.

The following theorem can be derived from the generalized poincare conjecture and Brown's collar theorem. We refer to \cite[\S 3.2]{GKL} for details and history.

 \begin{thm}\label{thm:poincare} 
 Let $X$ be a compact $n$-dimensional topological manifold with boundary, such that $\partial X$ is homeomorphic to an $(n-1)$-dimensional sphere and $X - \partial X$ is homeomorphic to an $n$-dimensional open ball. Then $X$ is homeomorphic to an $n$-dimensional closed ball $D^n$.
 \end{thm}

 
 
\section{$J$-Richardson varieties}

\subsection{The partial order $\leJ$ on $W$}\label{sec:poset}

Following \cite{BD} and \cite[\S 2.3 \& Proposition 4.6]{BH21}, we define the $J$-twisted length ${}^J \ell$ and the $J$-twisted Bruhat order $\leJ$ on $W$ as follows. For $w \in W$, $${}^J \ell(w)=\ell({}^J w)-\ell(w_J).$$ For $w, w' \in W$, $w' \leJ w$ if there exists $u \in W_J$ with $w_J \le w'_J u \i, u\, {}^Jw' \le {}^Jw$. This is a special case of the twisted Bruhat order consider in \cite{Dyer}. We say $w' \lessJ w$ if $w' \leJ w$ and $w' \neq w$. It follows from \cite[Proposition~1.7]{Dyer} that the poset $(W, \leJ)$ is graded. By \cite[Proposition~1.1]{Dyer}, if $v \leJ w$, then any maximal chain is of length $\Jell(w) - \Jell(v)$. 

We define the poset 
\[
\JQ= \{(v,w ) \in W \times W \vert v \leJ w\}, \text{ where } (v',w') \,\preceqJ (v,w), \text{ if }v \leJ v' \leJ w' \leJ w.
\]
We also define a poset $\JHQ = \JQ \bigsqcup\{\hat{0}\}$, where $\hat{0}$ is the new minimal element.
\begin{prop}\label{lem:JQ}
\begin{enumerate}
	\item The poset $(W, \leJ)$ graded, thin and shellable. In particular, any of the convex interval of $W$ is grade, thin and shellable.
	\item The poset $(\JHQ, \preceqJ)$ is graded, thin and shellable. In particular, any of the convex interval of $\JHQ$ is grade, thin and shellable.
\end{enumerate}
\end{prop}

\begin{proof}
It follows from \cite[Proposition~1.7 \& 2.5 \& 3.9]{Dyer} that the poset $(W, \leJ)$ is grade, thin and EL-shellable. 

We next equip $W \times W$ with the partial order $\preceqJ$ such that $(v_1, w_1) \preceq (v_2, w_2)$ if $ v_1 \leJ v_2$ and  $w_2 \leJ w_1$. It follows as a special case of $(1)$ that $(W \times W, \preceqJ)$  is grade, thin and EL-shellable. Note that poset ${}^J Q$ is a closed interval in $(W \times W, \preceq)$. Thus the poset ${}^JQ$ is graded, thin, and shellable.

It is easy to see that $(\JHQ, \preceqJ)$ is graded and thin.  The EL-shellability of $\JHQ$ can be proved similar to \cite{W07} (see also \cite[\S 4]{BH21}). 
\end{proof}

\subsection{The $J$-Richardson varieties} \label{sec:JRichardson}
We follow \cite[\S 2.3]{BH21} to introduce the $J$-Richardson varieties. Let $J \subset I$. Let $B^{\pm}_J=L_J \bigcap B^{\pm}$ and $U^{\pm}_J=L_J \bigcap U^{\pm}$. Set
\[
\JB^+ = B^-_J \ltimes U_{P^+_J}, \quad\JB^- = B^+_J \ltimes U_{P^-_J}.
\]

We set $\JU^+=U^-_J U_{P^+_J}$ and $\JU^-=U^+_J U_{P^-_J}$. Then $\JU^\pm$ is the unipotent radical of $\JB^\pm$. For $v, w \in W$, we define, respectively, the $J$-Schubert cell, the opposite $J$-Schubert cell and the open $J$-Richardson variety by 
$$
  {\OJCB}_w =\JB^+  \dot{w} B^+/B^+, \quad  \OJCB^v=\JB^-  \dot{v} B^+/B^+,\quad \OJCB_{v, w}= {\OJCB}_w  \bigcap \OJCB^v.
$$

By \cite[Proposition 2.4]{BH21},  ${\OJCB}_{v, w} \neq \emptyset$ if and only if $v \leJ w$. We have the decomposition 
$$
\CB=\bigsqcup_{v \leJ w}  {\OJCB}_{v, w}.
$$ 
By \cite[Theorem 4]{BD}, the Zariski closure of $ {\OJCB}_{v, w}$, denoted by  ${\JCB}_{v, w}$, is contained in $  \JCB_w \bigcap  \JCB^v=\bigsqcup_{v \leJ v' \leJ w' \leJ w}  {\OJCB}_{v', w'}$. We will show later in Proposition \ref{prop:Jcl} that $  {\JCB}_{v, w}$ equals $ \JCB_w \bigcap \JCB^v$. 

If $J=\emptyset$, then $v \leJ w$ if and only if $v \le w$. In this case, ${\OJCB}_{v, w}=\CB_{v, w}$. If $J=I$, then $v \leJ w$ if and only if $w \le v$. In this case, $ {\OJCB}_{v, w}=\CB_{w, v}$. 

\subsection{Some isomorphisms on ${\CB}$} \label{sec:Jcr}

For any $r \in W$, we have isomorphisms
\begin{align*}
( \dot{r} U^-  \dot{r}^{-1} \bigcap {}^{J} U^{+} )\times  ( \dot{r} U^-  \dot{r}^{-1} \bigcap {}^{J}U^{-})   &\longrightarrow   \dot{r} U^-  \dot{r}^{-1},   \quad
(g_1, g_2)  \mapsto g_1 g_2;  \\
   ( \dot{r} U^-  \dot{r}^{-1} \bigcap {}^{J} U^{-} )    \times ( \dot{r} U^-  \dot{r}^{-1} \bigcap {}^{J} U^{+} ) &\longrightarrow \dot{r} U^-  \dot{r}^{-1} ,   \quad
   (h_1, h_2)  \mapsto h_1 h_2.
\end{align*}

We define morphisms of ind-varieties
\begin{gather*}
\Jsigma_{r,-}: \dot{r} U^-  \dot{r}^{-1} \rightarrow   \dot{r} U^-  \dot{r}^{-1} \bigcap {}^{J}U^{-}, \qquad g_1 g_2  \mapsto g_2,\\
 \Jsigma_{r,+}:   \dot{r} U^-  \dot{r}^{-1} \rightarrow   \dot{r} U^-  \dot{r}^{-1} \bigcap {}^{J}U^{+},  \qquad h_1 h_2  \mapsto h_2.
\end{gather*}

We have the following isomorphism as a special case of \cite[Lemma~2.2]{KWY}:
\begin{equation}\label{eq:Jsigma}
\Jsigma_r=(\Jsigma_{r, +}, \Jsigma_{r, -}): \dot{r} U^-  \dot{r}^{-1} \xrightarrow{\sim}   (\dot{r} U^-  \dot{r}^{-1} \bigcap {}^{J} U^{+}) \times (\dot{r} U^-  \dot{r}^{-1} \bigcap {}^{J}U^{-}).
\end{equation}
  
By \eqref{eq:Jsigma}, for any $r \in W$, the map $g \dot{r} B^+/ B^+ \mapsto \Big( \Jsigma_{r,+}(g)  \dot{r}  \cdot B^+ / B^+,  \Jsigma_{r,-}(g) \dot{r} \cdot B^+ /B^+ \Big)$ for $g \in \dot{r} U^- \dot{r}^{-1}$ defines an isomorphism  
\begin{align}
\Jc_{r}=(\Jc_{r,+}, \Jc_{r,-}):  \dot{r} U^- B^+ /B^+ \xrightarrow{\sim}  {}^{J}\!\mathring{\CB}_{r} \times {}^{J}\!\mathring{\CB}^{r} \label{eq:atlasJ}.
\end{align}

The  map $\Jc_{r}$ sends $\OJCB_{v,w} \bigcap   (\dot{r} B^- B^+ /B^+)$ to ${}^{J}\!\mathring{\CB}_{v,r} \times {}^{J}\!\mathring{\CB}_{r,w}$ for any $v \leJ w$. 
The isomorphism in \eqref{eq:atlasJ} restricts to an isomorphism
\begin{equation}\label{eq:atlasJ1}
	\OJCB_{v,w} \bigcap   (\dot{r} B^- B^+ /B^+) \xrightarrow{\sim} {}^{J}\!\mathring{\CB}_{v,r} \times {}^{J}\!\mathring{\CB}_{r,w}.
\end{equation}
This also shows that 

(a) $\OJCB_{v,w} \bigcap   (\dot{r} B^- B^+ /B^+)  \neq \emptyset$ if and only if $ v \leJ r \leJ w$.

\subsection{Some general results}\label{sec:keylemma}
Note that the  map ${}^J \!c_{r}$ in \eqref{eq:atlasJ} is defined over $\BR$.

\begin{lem}\label{lem:key}
Let $ v\leJ u \leJ w$. Let $Y\subset \OJCB_{v,w}(\BR)$ with $\dot{u} B^+ /B^+ \in \overline{Y}$. Then  
\begin{enumerate}
\item for any $v \le v' \le u$, we have $\overline{Y} \bigcap \OJCB_{v',u}  = \overline{ {}^J\!c_{u,+}(Y \bigcap (\dot{u} B^- B^+ /B^+))}
 \bigcap \OJCB_{v',u} $; 
 \item for any $ u \le w' \le w$, we have $\overline{Y} \bigcap \OJCB_{u, w'}  = \overline{ {}^J\!c_{u,-}(Y \bigcap (\dot{u} B^- B^+ /B^+))}
 \bigcap \OJCB_{u, w'} $.
\end{enumerate}

\end{lem}
\begin{remark}\label{remark:Zariski}
Similar results hold if the Hausdorff closure is replaced by the Zariski closure. 
\end{remark}

\begin{proof}
Set $Y'=Y \bigcap (\dot{u} B^- B^+ /B^+)$. Then $\overline{Y} \bigcap (\dot{u} B^- B^+ /B^+)=\overline{Y'} \bigcap (\dot{u} B^- B^+ /B^+)$. We have the following isomorphism via restriction
\[
	 \overline{Y'} \bigcap \dot{u} B^- B^+ /B^+ \cong \overline{{}^J\!c_{u,+}(Y')} \times \overline{{}^J\!c_{u,-}(Y')}.
\]

Since $\dot{u} B^+ /B^+ \in \overline{Y} \bigcap (\dot{u} B^- B^+ /B^+) \subset \overline{Y'}$, we must have $\dot{u} B^+ /B^+ \in  \overline{{}^J\!c_{u,+}(Y')}$ and $\dot{u} B^+ /B^+ \in \overline{{}^J\!c_{u,-}(Y')}$. Since the isomorphism is stratified, we have 
\[
	 \overline{Y'} \bigcap \OJCB_{v',u}  \cong (\overline{ {}^J\!c_{u,+}(Y')}
 \bigcap \OJCB_{v',u}) \times \dot{u} B^+ /B^+ \cong  \overline{ {}^J\!c_{u,+}(Y')}
 \bigcap \OJCB_{v',u}.
\]
The composition is actually the identity map. Now part (1) follows. Part (2) is proved in the same way. 
\end{proof}
 
\begin{lem}\label{lem:key2}
Let $ v\leJ u \leJ w$ and $Y$ be a connected component of $\OJCB_{v,w}(\BR)$. If  $Y\subset \dot{u} B^- B^+ /B^+$, then 
	\begin{enumerate}
	\item  $\dot{u} B^+ /B^+ \in \overline{Y}$;
	\item $\overline{ Y}
 \bigcap \OJCB_{v,u}  = {}^J\!c_{u,+}(Y)$ is a connected component of $\OJCB_{v,u}(\BR)$;
	\item $\overline{ Y}
 \bigcap \OJCB_{u,w}  = {}^J\!c_{u,-}(Y) $ is a connected component of $\OJCB_{u,w}(\BR)$.
	\end{enumerate}
\end{lem}

\begin{proof}
Let $\mu$ be a dominant regular coweight. Then for any $g \in U^-$, we have $ \lim_{t \to 0} \mu(t) g \mu(t)^{-1} = 1$. Set $\mu' = u(\mu)$. Thus $\lim_{t \to 0} \mu'(t) \dot u  g B^+/ B^+=\dot u B^+/ B^+$ for any $g \in U^-$. Since $Y$ is a connected component of $\OJCB_{v,w}(\BR)$, it is stable under the action of $\mu'(\Rp)$. This shows that $\dot{u} B^+ /B^+ \in \overline{Y}$. Part (1) is proved. 

Thanks to \eqref{eq:atlasJ1},  we see that $ {}^J\!c_{u,+}(Y)$ is a connected component in $ \OJCB_{v,u}(\BR) $. Therefore ${}^J\!c_{u,+}(Y)$ is closed in $ \OJCB_{v,u}  $ and thus equals to $\overline{ {}^J\!c_{u,+}(Y)} \bigcap \OJCB_{v,u} $.  Hence by Lemma~\ref{lem:key}, we have ${}^J\!c_{u,+}(Y) =\overline{ Y}
 \bigcap\OJCB_{v,u}  = \overline{ {}^J\!c_{u,+}(Y)}
 \bigcap \OJCB_{v,u} $. We proved part (2). Part (3) can be proved similarly. 
\end{proof}

Now we prove the main result of this section. 

\begin{thm}\label{thm:product}
Let $Y_{v,w} $ be a connected component  of $\OJCB_{v,w}(\BR)$.  We define 
\[
	Y_{v', w'}  = \overline{Y_{v,w}}  \bigcap \OJCB_{v', w'},  \text{ for any $ v \leJ v'\leJ w' \leJ w$}.
\]
Assume that for any $ v \leJ v'\leJ u \leJ w' \leJ w$, we have $Y_{v', w'} \subset  \dot{u} B^- B^+ /B^+$.
Then the map  ${}^J\!c_{u}$ restricts to an isomorphism 
		\[
			 {}^J\!c_{u} : Y_{v', w'} \cong Y_{v', u} \times Y_{u, w'}.
		\]
\end{thm}

\begin{remark}
	We refer to the isomorphism above as a {\it product structure}.
\end{remark}

\begin{proof}
The proof consists of the following steps  
\begin{enumerate} 
\item[(i)] for any $v \leJ u \leJ w$, we have ${}^J\!c_{u} : Y_{v, w} \cong Y_{v, u} \times Y_{u, w}$;
\item[(ii)]  for any $v \leJ v' \leJ u \leJ w$, we have ${}^J\!c_{u} : Y_{v', w} \cong Y_{v', u} \times Y_{u, w}$;
\item[(iii)] for any $v \leJ u \leJ w' \leJ w$, we have ${}^J\!c_{u} : Y_{v, w'} \cong Y_{v, u} \times Y_{u, w'}$;
\item[(iv)]  for any $v \leJ v'\leJ u \leJ  w' \leJ w$, we have ${}^J\!c_{u} : Y_{v', w'} \cong Y_{v', u} \times Y_{u, w'}$.
\end{enumerate}

Thanks to Lemma~\ref{lem:key2} (at $u$), we see 
\[
Y_{u, w} =  {}^J\!c_{u,-}(Y_{v,w}) \quad \text{ and }  \quad  Y_{v, u} =  {}^J\!c_{u,+}(Y_{v,u}), \text{ for any } v \leJ u \leJ w.
\]
Part (i) follows. We also see that $Y_{u, w}$ and $Y_{v, u}$ are connected components of $ \OJCB_{u, w}(\BR) $ and $ \OJCB_{v, u}(\BR) $, respectively. 

	Let $u = v'$ in part (i). Applying Lemma~\ref{lem:key2} (at $w'$) to $Y_{v', w}$, we obtain that 
\[
	\overline{Y_{v', w}} \bigcap \OJCB_{v',w'} =  {}^J\!c_{w',+}(Y_{v', w}).
\]

We finally apply Lemma~\ref{lem:key} (at $v'$) to obtain that 
\begin{align*}
 Y_{v',w'} &= \overline{Y_{v,w}} \bigcap  \OJCB_{v',w'}   =   \overline{{}^J\!c_{v',-}(Y_{v, w})} \bigcap  \OJCB_{v',w'} 
 =  \overline{Y_{v', w}} \bigcap \OJCB_{v',w'}  \\
 &=  {}^J\!c_{w',+}(Y_{v', w}) = {}^J\!c_{w',+}( {}^J\!c_{v',-}(Y_{v,w})).
\end{align*}
One may prove similarly that $Y_{v', w'} = {}^J\!c_{v',-}   ( {}^J\!c_{w',+} (Y_{v,w}) )$, $Y_{w', w} = {}^J\!c_{w',-} (  {}^J\!c_{v',-}(Y_{v,w}) )$ and $Y_{v, v'} = {}^J\!c_{v',+} (  {}^J\!c_{w',+}(Y_{v,w}) )$.

Now part (ii) and (iii) follow. We also see that $Y_{v', w'}$ is a connected component of $\OJCB_{v',w'}(\BR) $. 

Let $ v' \leJ v'' \leJ w'' \leJ w$. Thanks to Lemma~\ref{lem:key}, we have 
\begin{align*}
 \overline{Y_{v', w}} \bigcap  \OJCB_{v'',w''}   &=   \overline{{}^J\!c_{v'',-}(Y_{v', w})} \bigcap  \OJCB_{v'',w''} 
 =  \overline{Y_{v'', w}} \bigcap \OJCB_{v'',w''}  \\
 &=  {}^J\!c_{w'',+}(Y_{v'', w}) = {}^J\!c_{w'',+}( {}^J\!c_{v'',-}(Y_{v,w}))
\\
& = Y_{v'', w''}.
\end{align*}

Now we see that $Y_{v', w}$ satisfies the same assumptions as $Y_{v,w}$, hence (i), (ii), (iii) apply to the new space $\overline{Y_{v', w}}$. Now part (iii) for the space $\overline{Y_{v', w}}$ is the same as part (iv) for the original space $\overline{Y_{v,w}}$. In particular, we have 
\begin{equation}\label{eq:Yclosure}
\overline{Y_{v', w'}} \bigcap  \OJCB_{v'',w''}  = Y_{v'', w''}.
\end{equation}

We hence finish the proof.
\end{proof}

Let us draw some consequences from Theorem~\ref{thm:product} and its proof. 

\begin{cor}\label{cor:product}
Retain the assumptions in  Theorem~\ref{thm:product}.
	\begin{enumerate}
		\item For any $ v \leJ v'\leJ w' \leJ w$, the subspace $Y_{v', w'}$ is a connected component in $\OJCB_{v', w'}(\BR)$.
		\item For any $ v \leJ v'\leJ w' \leJ w$, we have 
\[
	\overline{Y_{v',w'}} =\bigsqcup_{v' \leJ v''\leJ w'' \leJ w'}Y_{v'',w''}.
\]
		\item We have $Y_{v', w'} \cong \Rp^{\Jell(w') - \Jell(v')}$.
	\end{enumerate}
\end{cor}
\begin{proof}
(1) has been establish in the proof of Theorem \ref{thm:product}. By \S \ref{sec:JRichardson} and \eqref{eq:Yclosure}, we have
\[
\overline{Y_{v',w'}} = \!\!\! \!\!\! \bigsqcup_{v' \leJ v''\leJ w'' \leJ w'} \!\!\!\!\!\!  (\overline{Y_{v',w'}}  \bigcap \OJCB_{v'',w''})  =\!\!\!  \!\!\! \bigsqcup_{v' \leJ v''\leJ w'' \leJ w'}\!\!\! \!\!\! Y_{v'',w''}.
\]
Part (2) is proved.

We show (3). It follows from  \cite[Theorem 5 \& Corollary 6]{BD} that 
	\[
		 \OJCB_{v', w'}(\BR) \cong \BR^\times, \text{ if } v' \leJ w' \text{ and } \Jell(w') - \Jell(v') =1.
	\]
	In this case, by part (1) we have $Y_{v', w'} \cong \Rp$. Recall \S \ref{sec:poset} that the poset $(W, \leJ)$ is graded. Now the general case follows from Theorem~\ref{thm:product} by induction on $\Jell(w') - \Jell(v')$. 
\end{proof}
 

\section{A regular CW complex}
 
The main result of this section is the following theorem.

\begin{thm}\label{thm:regularCW} 
We fix $v \leJ w$. Let $Y_{v,w} $ be a connected component  of $\OJCB_{v,w}(\BR)$.  We define 
\[
	Y_{v', w'}  = \overline{Y_{v,w}}  \bigcap \OJCB_{v', w'}(\BR) \text{ for any $ v \leJ v'\leJ w' \leJ w$}.
\]
Assume that for any $v \leJ v'\leJ u \leJ w' \leJ w$, we have $Y_{v', w'} \subset  \dot{u} B^- B^+ /B^+$. 
Then $\overline{Y_{v,w}}=\bigsqcup_{(v', w')} Y_{v', w'}$ is a regular CW complex homeomorphic to a closed ball of dimension $\Jell(w) - \Jell(v)$.  
\end{thm}
 
\subsection{Links}\label{sec:link} 
Links can be defined for arbitrary Whitney stratified spaces; see \cite[Definition~4.25]{Her}. We shall not consider this abstract definition here, but follow \cite[Theorem~1.2]{FS} and \cite[\S 3.1]{GKL} instead.  
 
We denote by $X^{++}$ the set of dominant regular  weights of $G$. For $\lambda \in {X}^{++}$, we denote by ${{V}}^\lambda$ the highest weight simple ${G}$-module over $\BC$ with highest weight $\lambda$. Let ${\eta}_{\lambda}$ be a highest weight vector. We denote by ${{V}}^\lambda (\BR)$ the $\BR$-subspace of ${{V}}^\lambda$ spanned by the canonical basis.

For any $v' \in W$ with $(v', w) \in \JQ$, we consider the embedding 
 \begin{equation}\label{eq:lk}
 	 \OJCB^{v'}(\BR) \xrightarrow{(\dot{v}')^{-1} \cdot - }  U^- B^+ /B^+ \cong U^- \xrightarrow{u \mapsto u \cdot \eta_{\lambda}} {{V}}^\lambda.
 \end{equation}

The image of $ \OJCB^{v'}(\BR)  \bigcap  \JCB_{v,w}$ lies in a finite-dimensional subspace $L^{v'} \subset  {{V}}^\lambda (\BR)$; cf. \cite[Theorem~5]{BD}. We identify  $\OJCB^{v'}(\BR) $ with the image. We equip $L^{v'}$ with the standard Euclidean norm with respect to the canonical basis.  We define the links\footnote{In fact, the definition of links is independent of the choice of  $\lambda \in X^{++}$ up to a stratified homeomorphism. We do not use this fact in this paper.} $Lk_? (\overline{Y_{v,w}})$  (via the embedding above) by
\begin{align*}
	 Lk_{v',w'}(\overline{Y_{v,w}}) &= Y_{v', w'} \bigcap \{ || x|| =1 \vert x \in L^{v'} \}, \text{ for any $w'$ such that }  v' \lessJ w' \leJ w, \\
	 Lk_{v'}(\overline{Y_{v,w}}) &= \bigsqcup_{v' \lessJ w' \leJ w}Lk_{v',w'}(\overline{Y_{v,w}}) = \overline{Y_{v,w}}    \bigcap \{ || x|| =1 \vert x \in L^{v'} \}. 
\end{align*}

We simply write $Lk_? = Lk_? (\overline{Y_{v,w}})$ if there is no confusing. 

For any dominant regular coweight $\mu$, we consider the natural $\BR^\times$-action on $\CB(\BR)$ via the coweight $v'(\mu)$. Note that the action is compatible with the  $\BR^\times$-action on $ {{V}}^\lambda(\BR)$ via the  dominant regular coweight  $\mu$ through the embedding \eqref{eq:lk}. We shall abuse notations and denote both actions by $\vartheta_\mu$.

It is clear $\OJCB_{v', w'}(\BR)$  is stable under the action of $\vartheta_{\mu}$ and any connected component of $\OJCB_{v',w'}(\BR)$ is stable under the action of $\Rp$. This defines a contractive flow on the space $L^{v'}$ in the sense of \cite[Definition~2.2]{GKL}, since $\mu$ is dominant regular. 

The following results are proved in \cite[Lemma~3.4 \& Proposition~3.5]{GKL}.

	(a) For any $x \in \overline{Y_{v,w}} \bigcap  \OJCB^{v'}(\BR)$ such that $x \not \in Y_{v',v'}=\dot v' B^+/B^+$, there is a unique $t_1(x) \in \Rp$ such that $\vartheta_\mu(t_1(x))x \in Lk_{v'}$. Moreover, the map $x \rightarrow t_1(x) $ is continuous. 
	
	(b) We have a stratified isomorphism $\overline{Y_{v,w}} \bigcap  \OJCB^{v'}(\BR)  \cong \mathrm{Cone}(Lk_{v'})$ such that $Y_{v',v'}$ maps to the cone point and $Y_{v', w'}  \cong Lk_{v',w'} \times \Rp$, $x \mapsto (\vartheta_\mu(t_1(x))x, 1/t_1(x) )$ for $w' \neq v'$. 
Here for any topological space $A$, the cone over $A$ is defined by $\mathrm{Cone}(A) = (A \times \BR_{\ge 0}) / (A \times \{0\})$. 

We denote by $D^n$ the closed ball of dimension $n$. Note that $\mathrm{Cone}(D^n) \cong \BR^{n} \times \BR_{\ge 0}$.

 \begin{prop}\label{prop:lk}
	For any $ v' \, {}^J\!\!\!<u \leJ w$, we have stratified isomorphisms
	\[
		 Lk_{v'} \bigcap \dot{u} U^- B^+/B^+ = \!\!\! \bigsqcup_{u \leJ w' \leJ w}\!\!\! Lk_{v',w'} \cong Lk_{v', u} \times  (Y \bigcap  \OJCB^{u}(\BR)) \cong Lk_{v', u} \times  \mathrm{Cone}(Lk_{u}).
	\] 
 \end{prop}	
 
 \begin{proof}We write  $\vartheta$ for  $\vartheta_{\mu}$.  Recall the assumption that $Y_{v', w' } \subset \dot{u} U^- B^+/B^+$ for $v' \leJ u \leJ w'$. Hence $Lk_{v'} \bigcap \dot{u} U^- B^+/B^+ =  \bigsqcup_{u \leJ w' \leJ w} Lk_{v',w'} $. The last isomorphism follows from \S\ref{sec:link} (b). We construct the second isomorphism. 
 
We first define a morphism $\alpha$ as the following composition 
\[
Lk_{v',u} \times (\overline{Y_{v,w}} \bigcap  \OJCB^{u}(\BR)) \hookrightarrow Y_{v',u} \times (\overline{Y_{v,w}} \bigcap  \OJCB^{u}(\BR)) \xrightarrow{\Jc_u^{-1} }\!\!\! \bigsqcup_{u \leJ w' \leJ w}\!\!\! Y_{v',w'}  \xrightarrow{\pi} \!\!\! \bigsqcup_{u \leJ w' \leJ w}\!\!\! Lk_{v',w'}.
\] Here the second map comes from Theorem~\ref{thm:product}. We next construct the inverse. 
 We then define a morphism $\beta$ as follows 
 \[
 	 \!\!\! \bigsqcup_{u \leJ w' \leJ w}\!\!\! Lk_{v',w'} \hookrightarrow  \!\!\! \bigsqcup_{u \leJ w' \leJ w}\!\!\! Y_{v',w'} \xrightarrow{\Jc_u} Y_{v',u} \times (\overline{Y_{v,w}} \bigcap  \OJCB^{u}(\BR)) \xrightarrow{\phi}  Lk_{v',u} \times (\overline{Y_{v,w}} \bigcap  \OJCB^{u}(\BR)),
 \] where $\phi(x, y)=(\vartheta(t_1(x))x, \vartheta(t_1(x)) y)$. 
 We claim $\alpha$ and $\beta$ give the desired isomorphism. The compatibility with the stratification is clear. We show that they are inverse to each other. 
 
 We first show $\beta \circ \alpha = \id$. Let $(x,y) \in Lk_{v',u} \times (\overline{Y_{v,w}} \bigcap  \OJCB^{u}(\BR))$. Then let $z = \Jc^{-1}_u((x,y))$. Then $\alpha(x,y) = \pi(z)=\vartheta(t_1(z))z$. Since $\Jc_u$ is $\Rp$-equivariant, we have $\Jc_u( \vartheta(t_1(z))z) = (\vartheta(t_1(z))   x,\vartheta(t_1(z))y)$. By the uniqueness in  \S\ref{sec:link} (a), we have $ \vartheta(t_1( \vartheta(t_1(z))   x)) x = x \in Lk_{v',u}$, hence $t_1( \vartheta(t_1(z))   x) = t_1(z)^{-1}$. Therefore we have $\phi(\vartheta(t_1(z))   x,\vartheta(t_1(z))y) = \beta(z)= (x,y)$.  
 
 We next show $\alpha \circ \beta = \id$. Let $z \in \! \bigsqcup_{u \leJ w' \leJ w}\! Lk_{v',w'} $ and $(x,y) = \Jc_u(z)$. Then $\beta(z) = (\vartheta(t_1(x))x, \vartheta(t_1(x)) y)$. Then since $\Jc_u$ is $T$-equivariant, we obtain $\Jc_u^{-1} (\beta(z) ) = \vartheta(t_1(x)) z$. Then by the uniqueness in  \S\ref{sec:link} (a), we must have $\pi(\vartheta(t_1(x)) z) = z$, that is $\alpha \circ \beta(z) = z$.
 
 We finish the proof.
 \end{proof}

 \begin{cor}\label{cor:lkcell}
Let $v' \,{}^J\!\!\!< w'$. We have 
 \[
 	 Lk_{v',w'} \cong \Rp^{\Jell(w') - \Jell(v') -1}.
 \]
 \end{cor}
 \begin{proof}
 Thanks to \S\ref{sec:link} (b) and Proposition~\ref{prop:lk}, it suffices to show $Lk_{v',w'}$ is a point when $\Jell(w') - \Jell(v') =1$. The latter statement follows from Theorem~\ref{thm:product} and direct computation. 
 \end{proof}
 
 \begin{prop}\label{prop:lkregular}
 For $v \leJ v' \lessJ w$, $Lk_{v'}=\sqcup_{v' \lessJ w' \leJ w} Lk_{v', w'}$ is a regular CW complex homeomorphic to a closed ball of dimension $\Jell(w) - \Jell(v') -1$. 
 
 \end{prop}
 
 \begin{proof}
 	We prove by induction on $\Jell(w) - \Jell(v')$. When $\Jell(w) - \Jell(v') =1$,  $Lk_{v'}$ is a point by Corollary~\ref{cor:lkcell}. In the induction process, we shall consider $Lk_? (\overline{Y_{v,w'}})$ for $w' \leJ w$ as well. Note that $Lk_{v'} (\overline{Y_{v,w'}})$ is a subspace of $Lk_{v'} (\overline{Y_{v,w}})$ and $Lk_{v', w''} (\overline{Y_{v,w'}}) = Lk_{v', w''} (\overline{Y_{v,w'}})$ for $v' {}^J\!\!\!< w'' \leJ w'$.
	
	We first show
	
	(a) $Lk_{v'}$ is a topological manifold with boundary  $\partial Lk_{v'} = \bigsqcup_{v' {}^J\!< w' {}^J\!< w} Lk_{v', w'}$.

	We have $Lk_{v', w} = Lk_{v'} \bigcap \dot{w} U^-B^+/B^+ \cong \Rp^{\Jell(w) - \Jell(v') -1}$. Now for any $u$ with $v' \lessJ u \lessJ w$, we apply the stratified isomorphism in Proposition~\ref{prop:lk}. We have 
	\begin{align*}
		 Lk_{v'} \bigcap \dot{u} U^- B^+/B^+ &= \!\!\! \bigsqcup_{u \leJ w' \leJ w}\!\!\! Lk_{v',w'} \cong Lk_{v', u} \times  \mathrm{Cone}(Lk_{u})\\
		 &\cong \Rp^{\Jell(u) - \Jell(v') -1} \times \mathrm{Cone}(D^{\Jell(w) - \Jell(u) -1}) \\
		 &\cong  \Rp^{\Jell(w) - \Jell(v') -2} \times \BR_{\ge 0}.
	\end{align*}
	 Here $Lk_{u} \cong D^{\Jell(w) - \Jell(u) -1}$ is obtained via the induction hypothesis since $\Jell(w) - \Jell(u) < \Jell(w) - \Jell(v')$. This shows that $Lk_{v'} $ is a topological manifold with boundary and $Lk_{v', u} $ lies on the boundary for $u \neq w'$. This proves (a).
	
	We next prove 
	
	(b) $ \partial Lk_{v'}= \bigsqcup_{v' {}^J\!< w' {}^J\!< w} Lk_{v', w'}$ is a regular CW complex homeomorphic to a sphere of dimension $\Jell(w) - \Jell(v') -2$. 
	
	 By induction hypothesis $Lk_{v'} (\overline{Y_{v,w'}})$ is a regular CW complex homeomorphic to a closed ball of dimension $\Jell(w) - \Jell(v') -1$, for any $v'  \lessJ w' \lessJ w$. Therefore  $\partial Lk_{v'} $ is a regular CW complex with the face poset $(\{ w' \vert v' \lessJ w' \lessJ w \}, \leJ)$. It is clear after adding a new minimal $\hat{0}$ and $\hat{1}$, the poset is $\{ w' \vert v'  \leJ w' \leJ w \}$, which is graded, thin, and shellable by Proposition~\ref{lem:JQ}. Hence by Theorem~\ref{thm:CW}, $ \partial Lk_{v'}$ is homeomorphic to a sphere of dimension $\Jell(w) - \Jell(v') -2$. This proves (b).
	
	The statement now follows from (a), (b) and Theorem~\ref{thm:poincare}.
 \end{proof}

\subsection{Proof of Theorem \ref{thm:regularCW}} Set $Y=\overline{Y_{v,w}}$. The outline of the proof is similar to the proof of Proposition~\ref{prop:lkregular}. We prove by induction on $\Jell(w) - \Jell(v)$. The base case then $\Jell(w) - \Jell(v) =0$ is trivial, since $Y_{v, v}=\dot v B^+/B^+$ is a single point.
 
We first show that 

(a) $Y$ is a topological manifold with boundary $\partial Y = \displaystyle\bigsqcup_{v \leJ v' \leJ w' \leJ w, (v,w) \neq (v',w')} Y_{v',w'}$.
	
The proof of the claim is divided into several cases depending on $(v',w')$. 
	\begin{enumerate}
		\item[(i)] It follows from Corollary~\ref{cor:product} that $Y_{v,w} \cong \Rp^{\Jell(w) - \Jell(v)}$ is the open cell. 
			\item [(ii)]	We consider the case $(v',w') = (v,v)$. By Proposition~\ref{prop:lkregular} and \S\ref{sec:link} (b), we have 
			\[
				Y \bigcap \dot{v}B^- B^+/B^+ =  Y \bigcap  \OJCB^{v} \cong \mathrm{Cone}(D^{\Jell(w) - \Jell(v) -1}), \quad Y_{v,v} \mapsto \text{ cone point}.
			\]
			This shows that $Y_{v,v}$ lies on the boundary. 
			\item [(iii)] We consider the case $(v',w') = (w,w)$. Similar to (ii), we can establish (via a variation of Proposition~\ref{prop:lkregular} and \S\ref{sec:link} (b)) 
			\[
				Y \bigcap \dot{w}B^- B^+/B^+ =  Y \bigcap  \OJCB_{w} \cong \mathrm{Cone}(D^{\Jell(w) - \Jell(v) -1}), \quad Y_{w,w} \mapsto \text{ cone point}.
			\]
			This shows that $Y_{w,w}$ lies on the boundary.
		\item[(iv)] We next consider the case when $v' \neq v$. We further assume $v' \neq w$, otherwise we are done by (iii). We apply the stratified isomorphism in Theorem~\ref{thm:product} to obtain
	\[
	Y \bigcap \dot{v}' B^-B^+/B^+ = \!\!\!  \!\!\! \bigsqcup_{v \leJ v'' \leJ v' \leJ w'' \leJ w}  \!\!\!   \!\!\! Y_{v'', w''} \cong   \!\!\! \bigsqcup_{v \leJ v'' \leJ v'} \!\!\!  Y_{v'', v'} \times   \!\!\!  \bigsqcup_{v' \leJ w'' \leJ w} \!\!\!  Y_{v', w''}
	\]
	Now induction applies to the spaces $\overline{ Y_{v'', v'}}$ and $\overline{Y_{v', w''}}$. In particular, they are both topological manifolds with the expect boundaries. Therefore 
	\begin{align*} 
	Y \bigcap \dot{v}' B^-B^+/B^+ &\cong \mathrm{Cone}(D^{\Jell(v') - \Jell(v) - 1}) \times \mathrm{Cone}(D^{\Jell(w) - \Jell(v') - 1}) \\
	&\cong  (\Rp^{\Jell(v') - \Jell(v) - 1} \times \BR_{\ge 0}) \times (\Rp^{\Jell(w) - \Jell(v') - 1} \times \BR_{\ge 0})  \\
	&\cong \Rp^{\Jell(w) - \Jell(v) - 1} \times \BR_{\ge 0}
	\end{align*}
     with $Y_{v',w'}$ lying on the boundary.
	\item [(v)] The final case $w' \neq w$ is similar to (iv).
	\end{enumerate}
	Now we finish the proof of (a).
	
	We next show that 
	
	(b) $\partial Y$ is a regular CW complex homeomorphic to a sphere of dimension $\Jell(w) - \Jell(v) -1$. 
	
	It follows by induction hypothesis that $\overline{Y_{v',w'}}$ is a regular CW complex homeomorphic to a closed ball of dimension $\Jell(w') - \Jell(v')$ if ${v \leJ v' \leJ w' \leJ w}$ and ${(v,w) \neq (v',w')}$. Hence $\partial Y$ is a regular CW complex with the face poset $\{ (v',w') \vert  { (v',w')\,  \preceqJ  (v,w), (v',w')  \neq (v,w) }  \} \subset \JQ$.
	By adding a new maximal element and a new minimal element we obtain the new poset $\{ (v',w') \in \JQ \vert  {(v',w')  \, \preceqJ (v,w)}  \} \bigsqcup \{\hat{0}\} \subset \JHQ$. 
	Thanks to Proposition~\ref{lem:JQ}, this is graded, thin and shellable. By Theorem~\ref{thm:CW}, $\partial Y$ is a regular CW complex homeomorphic to a sphere of dimension $\Jell(w) - \Jell(v) -1$. This finishes the proof of (b).
	
	The statement now follows from (a), (b) and Theorem~\ref{thm:poincare}.

 
\section{Main results}
 
We collect the main results of this paper. We shall first prove the main results on the ordinary totally nonnegative flag variety  $\CB_{\ge 0}$ in \S\ref{sec:classical}.  In \S\ref{sec:partial} and \S\ref{sec:JTP}, we will prove the main results on the  totally nonnegative partial flag variety and the $J$-totally nonnegative flag variety, respectively. The proofs rely on Proposition~\ref{prop:CPmanifold} and Proposition~\ref{prop:J} (both marked with $\club$) respectively, which will be established in \S\ref{sec:6} to \S\ref{sec:8}. 
 
\subsection{Ordinary total positivity}\label{sec:classical}
\subsubsection{Totally nonnegative part of $G$}\label{Rp-BH} 
We follow \cite{Lus-1} and \cite{Lu-2}. The generalization to Kac-Moody groups is straightforward. 
Let $U^+_{\ge 0}$ be the submonoid of $G$ generated by $x_i(a)$ for $i \in I$ and $a \in \Rp$ and $U^-_{\ge 0}$ be the submonoid of $G$ generated by $y_i(a)$ for $i \in I$ and $a \in \Rp$. Let $T_{>0}$ be the identity component of $T(\BR)$. Let $G_{\ge 0}$ be the submonoid of $G$ generated by $U^{\pm}_{\ge 0}$ and $T_{>0}$. By \cite[\S 2.5]{Lu-2}, $G_{\ge 0}=U^+_{\ge 0} T_{>0} U^-_{\ge 0}=U^-_{\ge 0} T_{>0} U^+_{\ge 0}$.

Let $w \in W$ and $w=s_{i_1} s_{i_2} \cdots s_{i_n}$ be a reduced expression of $w$. Set 
\begin{gather*}
U^+_{w, >0}=\{x_{i_1}(a_1) x_{i_2}(a_2) \cdots x_{i_n}(a_n) \vert a_1, a_2, \ldots, a_n \in \Rp\}; \\
U^-_{w, >0}=\{y_{i_1}(a_1) y_{i_2}(a_2) \cdots y_{i_n}(a_n) \vert a_1, a_2, \ldots, a_n \in \Rp\}. 
\end{gather*}

By \cite[Lemma 2.3 (b)]{Lu-2}, $U^{\pm}_{w, >0}$ is independent of the choice of reduced expressions of $w$. Moreover, by \cite[\S 2.5 (d) \& (e)]{Lu-2}, we have $U^{\pm}_{\ge 0}=\bigsqcup_{w \in W} U^{\pm}_{w, >0}$.

\subsubsection{Totally nonnegative flag varieties}
Let $\CB_{\ge 0}=\overline{U^-_{\ge 0} \cdot B^+}$ be the closure of $U^-_{\ge 0} \cdot B^+$ in $\CB$ with respect to the Hausdorff topology. For any $v \le w$, let 
\[
\CB_{v, w, >0}=\mathring{\CB}_{v, w} \bigcap \CB_{\ge 0}.
\] Hence $\CB_{\ge 0}=\bigsqcup_{v \le w} \CB_{v, w, >0}$. 

Let ${\bf w}=s_{i_1} s_{i_2} \cdots s_{i_n}$ be a reduced expression of $w \in W$. A subexpression of $\bf w$ is $t_{i_1} t_{i_2} \cdots t_{i_n}$, where $t_{i_j}=1$ or $s_{i_j}$ for any $j$. For any $v \le w$, there exists a unique positive subexpression ${\bf v}_+=t_{i_1} t_{i_2} \cdots t_{i_n}$ for $v$ in ${\bf w}$ in the sense of \cite[Lemma 3.5]{MR}. Following  \cite[Definition~5.1]{MR}, we set 
\[
 G_{\bf v_+, \bf w, >0}=\{g_1 g_2 \cdots g_n\vert  g_j=\dot s_{i_j}, \text{ if } t_{i_j}=1; \text{ and } g_j \in y_{i_j}(\Rp), \text{ if } t_{i_j}=s_{i_j}\}.
\]
Note that the obvious map $\Rp^{\ell(w) - \ell(v)} \rightarrow  G_{\bf v_+, \bf w, >0}$
is a homeomorphism.

By \cite[Theorem~11.3]{MR} for reductive groups and  \cite[Theorem~4.10]{BH20} for Kac-Moody groups, we have the following parametrization result. 

(a) Let $v \le w$. For any reduced expression ${\bf w}$ of $w$, the  map $g \mapsto g \cdot B^+$ gives a homeomorphism 
 \begin{equation}\label{eq:5.1}
 G_{\bf v_+, \bf w, >0} \cong \CB_{v, w, >0}.
   \end{equation}
In particular, $\CB_{v, w, >0} \cong \Rp^{\ell(w)-\ell(v)}$ is a topological cell.

\smallskip

We recall the monoid actions $\ast, \circ_l$ and $\circ_r$ of $W$ in \cite[\S 5]{BH20}. We have $s_i\ast w = \max\{w, s_iw\}$,  $s_i \circ_l w = \min \{w, s_iw\}$, and $w \circ_r s_i = \min \{w, ws_i\}$ for any simple reflection $s_i$.

\begin{lem}\label{lem-in-u}
Let $v \le r \le w$ in $W$. We have 
\[
	\CB_{v,w, >0} \subset \dot{r} B^- B^+/ B^+. 
\]
\end{lem}
\begin{proof}
We argue  by induction on $\ell(w)$.  Let $w' = s_{i_1} w$ with the reduced expression ${\bf w'} = s_{i_2} \cdots s_{i_n}$. Set $ v' = s_{i_1} \circ_l v$ and $r' = s_{i_1} \circ_l r$. It follows from \cite[Lemma~2]{He-0hecke} that $v' \le r' \le w'$.  
We divide the computation into several cases.

\begin{itemize}
	\item If $r \le s_{i_1}r$ or $r' =r$, then $v \le r \le s_{i_1}w$. In this case $g_1 \in y_{i_1}(\Rp)$ and $ \dot{r}^{-1} g_1   \dot{r} \in U^-(\BR)$.
	Hence $\dot{r}^{-1}  G_{\bf v_+, \bf w, >0}  \subset  U^-(\BR)  \dot{r}^{-1}  G_{{\bf v_+}, {\bf w'}, >0}$. 
	\item If $r \ge s_{i_1}r$ and $g_1 = \dot{s}_{i_1}$, then $v' = s_{i_1}v$. Therefore $\dot{r}^{-1}  G_{\bf v_+, \bf w, >0}  =  \dot{r}^{', -1}G_{{\bf v'_+}, {\bf w'}, >0}$.
	\item Assume $r \ge s_{i_1}r$ and $g_1 \in y_{i_1}(\Rp)$. Then $G_{\bf v_+, \bf w, >0} = y_{i_1}(\Rp) G_{{\bf v_+}, {\bf w'}, >0}$. For any $a \in \Rp$, we have $\dot{s}^{-1} _{i_1}y_{i_1}(a) =  \alpha_{i_1}^\vee(a^{-1}) y_{i_1}(-a) x_{i_1}(a^{-1})$. Then we have $x_{i_1}(a^{-1}) G_{{\bf v_+}, {\bf w'}, >0} \subset G_{{\bf v'_+}, {\bf w'}, >0}  B^+$ by \cite[Proposition~5.2]{BH20}. Therefore $\dot{r}^{-1}  G_{\bf v_+, \bf w, >0} \subset   B^-  (\dot{r}')^{-1}G_{{\bf v'_+}, {\bf w'}, >0} B^+$.
\end{itemize}

The statement then follows from inductive hypothesis on $w'$. 
\end{proof}

\subsubsection{Main results on $\CB_{\ge 0}$}

We apply results in Theorem~\ref{thm:product} and Theorem~\ref{thm:regularCW} for the case $J = \emptyset$ to prove the main result for $\CB_{\ge 0}$.

\begin{prop}\label{prop:CB} 
Let $v, w \in W$ with $v \le w$. Then
\begin{enumerate}
\item $\CB_{v, w, > 0 } $ is a connected component of $\OCB_{v, w}(\BR)$.
	\item We have 
$
	\CB_{v, w, \ge 0 } = \overline{ \CB_{v, w, > 0 }}= \bigsqcup_{v \le v' \le w' \le w} \CB_{v', w', >0}$.
\end{enumerate}

\end{prop}
\begin{proof}
We first consider the $v=1$ case. We have a commutative diagram 
\[
\xymatrix{
U^-_{w, >0} \ar[r] \ar[d]^-\cong & U^- \bigcap B^+ \dot w B^+ \ar[d]^-\cong \\
\CB_{1, w, >0} \ar[r] & \OCB_{1, w}.
}
\]
By \cite[\S 6.3]{Lu3}, $U^-_{w, >0}$ is a connected component of $U^-(\BR) \bigcap B^+ \dot w B^+$. Thus $\CB_{1, w, >0}$ is a connected component of $\OCB_{1, w}(\BR)$. 

Let $v_1 \le w_1 \le w$. By Lemma~\ref{lem-in-u}, $\overline{\CB_{1, w, >0}} \bigcap \OCB_{v_1, w_1} \subset \CB_{v_1, w_1, >0} \subset \dot u U^- B^+/B^+$ for any $u \in W$ with $v_1 \le u \le w_1$. Hence the assumption in Theorem~\ref{thm:product} is satisfied for $Y_{1, w}=\CB_{1, w, >0}$. By Corollary~\ref{cor:product}, $\overline{\CB_{1, w, >0}} \bigcap \OCB_{v_1, w_1} \subset \CB_{v_1, w_1, >0}$ is a connected component of $\OCB_{v_1, w_1}(\BR)$. By \eqref{eq:5.1}, $\CB_{v_1, w_1, >0}$ is connected. Hence $\CB_{v_1, w_1, >0}=\overline{\CB_{1, w, >0}} \bigcap \OCB_{v_1, w_1}$ and it is a connected component of $\OCB_{v_1, w_1}(\BR)$. Moreover, part (2)  for $v=1$ now follows from Theorem~\ref{thm:product}, Theorem~\ref{thm:regularCW} and
Corollary~\ref{cor:product} for $Y_{1, w}=\CB_{1, w, >0}$. 

Now we consider the general case. We have already prove that $Y_{v, w}=\CB_{v, w, >0}$ is a connected component of $\OCB_{v, w}(\BR)$. By Lemma~\ref{lem-in-u} and Corollary~\ref{cor:product}, for $v_1, w_1 \in W$ with $v \le v_1 \le w_1 \le w$, $\overline{\CB_{v, w, >0}} \bigcap \OCB_{v_1, w_1} \subset \CB_{v_1, w_1, >0}$ is a connected component of $\OCB_{v_1, w_1}(\BR)$. Hence $\overline{\CB_{v, w, >0}} \bigcap \OCB_{v_1, w_1}=\CB_{v_1, w_1, >0}$. Now part (2)  for arbitrary $v \in W$ now follows from Theorem~\ref{thm:product}, Corollary~\ref{cor:product} for $Y_{v, w}=\CB_{v, w, >0}$.  
\end{proof}

We have verified the assumption in Theorem~\ref{thm:product} for $Y_{v, w}=\CB_{v, w, >0}$. The following theorem follows from Theorem~\ref{thm:product}, Theorem~\ref{thm:regularCW}, Corollary~\ref{cor:product}, and Proposition~\ref{prop:lkregular} for  $Y_{v, w}=\CB_{v, w, >0}$.

\begin{thm}\label{thm:CB}  Let $v, w \in W$ with $v \le w$.   
\begin{enumerate}
	\item For any $u \in W$ with $v \le u \le w$, the map $c_u$ restricts to an isomorphism 
	\[
	 c_u: \CB_{v, w, > 0 }  \cong \CB_{v, u, >  0 }  \times \CB_{u, w,> 0 }.
	\]
	\item $\CB_{v, w, \ge 0 }$ is a regular CW complex homeomorphic to a closed ball of dimension $\ell(w) - \ell(v)$.
\end{enumerate}
\end{thm}
\begin{remark}
Part (2) of Theorem~\ref{thm:CB} proves \cite[Conjecture~10.2 (2)]{GKL}.  
\end{remark}

 
  
\subsection{Totally nonnegative partial flag varieties}\label{sec:partial}
  
\subsubsection{Partial flag varieties}
Let $K \subset I$ and $\CP_K=G/P^+_K$ be the partial flag variety. 
Let $Q_K=\{(v, w) \in W \times W^K \vert v \le w\}$. Define the partial order $\preceq$ on $Q_K$ by $(v', w') \preceq (v, w) $ if there exists $u \in W_K$ with $v \le v' u \le w' u \le w$. For any $(v, w) \in Q_K$, set 
$$\mathring{\CP}_{K, (v, w)}=pr_K(\OCB_{v, w}) \text{ and } \CP_{K, (v, w)}=pr_K(\CB_{v, w}),$$ where $pr_K: \CB \to \CP_K$ is the projection map. Then $\CP_{K, (v, w)}$ is the (Zariski) closure of $\mathring{\CP}_{K, (v, w)}$ in $\CP_K$. We call $\mathring{\CP}_{K, (v, w)}$ an {\it open projected Richardson variety} and $\CP_{K, (v, w)}$ a {\it closed projected Richardson variety}. Note that $pr_K: \OCB_{v, w} \rightarrow \mathring{\CP}_{K, (v, w)}$ is an isomorphism for $(v,w) \in Q_K$.

By \cite[Proposition 3.6]{KLS}, we have 
\begin{equation}\label{eq:Richardson}
	\CP_K=\bigsqcup_{(v, w) \in Q_K} \mathring{\CP}_{K, (v,w)} \quad \text{ and }\quad  {\CP}_{K, (v,w)}=\bigsqcup_{{(v', w') \in Q_K,}{ (v', w') \preceq (v, w)}}\mathring{\CP}_{K, (v',w')}.
\end{equation} 
\subsubsection{Total positivity on $\CP_K$} Let $\CP_{K, \ge 0}=\overline{U^-_{\ge 0} P^+_K/P^+_K}$ be the closure of $U^-_{\ge 0} P^+_K/P^+_K$ in $\CP_K$ with respect to the Hausdorff topology. It is easy to see that $\CP_{K, \ge 0} = pr_K(\CB_{\ge 0})$. For $(v,w) \in Q_K$, we further define $\CP_{K, (v,w), > 0} = \CP_{K, \ge 0} \bigcap \mathring{\CP}_{K, (v,w)}$.

\begin{prop}\label{prop:closure-p}
	Let $(v,w) \in Q_K$. Then we have 
	\begin{enumerate}
		\item $\CP_{K, (v, w), >0}=pr_K(\CB_{v, w, >0})$; 
		
		\item $\overline{ \CP_{K, (v,w), > 0}} = \bigsqcup_{(v', w') \preceq (v,w) \text{ in } Q_K}\CP_{K, (v',w'), > 0}$;
			\item $\CP_{K, (v,w), > 0}$ is a connected component of $\mathring{\CP}_{K, (v,w)}(\BR)$;
			\item we have
	\[
	\begin{cases}
		\CP_{K, (v,w), > 0} \subset \dot{r}U^- P^+_K / P^+_K, &\text{if } (r,r) \preceq (v,w) \in Q_K;\\
		\CP_{K, (v,w), > 0} \bigcap \dot{r}U^- P^+_K / P^+_K = \emptyset, &\text{otherwise. }
	\end{cases}
	\]

	\end{enumerate}
\end{prop}
 \begin{proof}
Let $v', w' \in W$ with $v' \le w'$. We write $w'=(w')^K \, w'_K$ with $(w')^K \in W^K$ and $w'_K \in W_K$. Set $v'_1=(w'_K) \i \circ_r v'$. Then $v'_1 \le (w')^K$. Set $v'_2=(v'_1) \i v'$. We fix a reduced expression ${\bf (w')^K}$ of $(w')^K$ and a reduced expression ${\bf w'_K}$ of $w'_K$. Then ${\bf (w')^K} {\bf w'_K}$ is a reduced expression of $w'$. Let $({\bf v'_1})_+$ be the positive subexpression of $v'_1$ in ${\bf (w')^K}$ and $({\bf v'_2})_+$ be the positive subexpression of $v'_2$ in ${\bf w'_K}$. It is easy to see that $({\bf v'_1})_+ ({\bf v'_2})_+$ is the positive subexpression of $v'$ in ${\bf (w')^K} {\bf w'_K}$. By definition, \begin{align*} pr_K(\CB_{v', w', >0}) & =pr_K(G_{({\bf v'_1})_+ ({\bf v'_2})_+, {\bf (w')^K} {\bf w'_K}, >0} \cdot B^+)=pr_K(G_{({\bf v'_1})_+, {\bf (w')^K}, >0} \cdot B^+) \\ &=pr_K(\CB_{v'_1, (w')^K, >0}) \subset \CP_{K, (v'_1, (w')^K)}. \end{align*}
	
In particular, $pr_K(\CB_{\ge 0})=\bigcup_{v' \le w' \text{ in } W} pr_K(\CB_{v', w', >0})=\bigcup_{v' \le w' \text{ in } W} pr_K(\CB_{v'_1, (w')^K, >0})$. Since $(v'_1, (w')^K) \in Q_K$, we have $pr_K(\CB_{\ge 0})=\bigcup_{(v', w') \in Q_K} pr_K(\CB_{(v', w'), >0})$. For any $(v', w') \in Q_K$, $pr_K(\CB_{(v', w'), >0}) \subset \CP_{K, (v', w')}$. Thus the union $\bigcup_{(v', w') \in Q_K} pr_K(\CB_{(v', w'), >0})$ is a disjoint union and $pr_K(\CB_{(v', w'), >0})=\CP_{K, (v', w'), >0}$ for any $(v', w') \in Q_K$.
	
Part (1) is proved. 
	
We have 
	\[
	\overline{ \CP_{K, (v,w), > 0}} =pr_K(\overline{\CB_{v, w, >0}})=  \bigsqcup_{v \le v'' \le w'' \le w} pr_K(\CB_{v'', w'', >0}).
	\]
	
Let $(v', w') \in Q_K$ with $(v', w') \preceq (v, w)$. Then there exists $u \in W_K$ such that $v \le v' u \le w' u \le w$. Let $u' \le u$ with $v' \ast u=v' u'$. Then $v \le v' u \le v' \ast u=v' u' \le w' u' \le w' u \le w$. We have $pr_K(\CB_{v' u', w' u', >0})=pr_K(\CB_{v', w', >0})=\CP_{K, (v', w'), >0}$. Thus $\CP_{K, (v', w'), >0} \subset \overline{\CP_{K, (v, w), >0}}$.  
	
On the other hand, for any $v' \le w'$ in $W$ with $v \le v' \le w' \le w$, we have $v \le v'=v'_1 v'_2 \le (w')^K v'_2 \le (w')^K w'_K=w' \le w$, where $v'_1$ and $v'_2$ are defined above. Thus $(v'_1, (w')^K) \preceq (v, w)$ and $pr_K(\CB_{v', w', >0})=\CP_{v'_1, (w')^K, >0}$. Part (2) is proved. 
	
Finally part (3) and part (4) follow from Theorem~\ref{thm:CB} and Lemma~\ref{lem-in-u}.
\end{proof}

\subsubsection{Main results on $\CP_{K, \ge 0}$}
We collect the main results on $\CP_{K, \ge 0}$ in this subsection. The proof relies on the following result, which will be proved in \S \ref{sec:pf2}.

\begin{prop}[$\club$]\label{prop:CPmanifold}
Let $(v,w) \in Q_K$. Then $\overline{ \CP_{K, (v,w), > 0}}$ is a topological manifold with boundary $\partial  \overline{ \CP_{K, (v,w), > 0}} = \bigsqcup_{(v',w') < (v,w) \in Q_K} \CP_{K, (v',w'), > 0}$.
\end{prop}

Now we prove the main result. 

\begin{thm}\label{thm:CPK}
Let $(v,w) \in Q_K$. Then $ \overline{ \CP_{K, (v,w), > 0}}= \bigsqcup_{(v', w') \preceq (v,w) \text{ in } Q_K}\CP_{K, (v',w'), > 0}$ is a regular CW complex homeomorphic to a closed ball of dimension $\ell(w)- \ell(v)$.
\end{thm}

\begin{remark}
This proves \cite[Conjecture~10.2 (3)]{GKL}.
\end{remark}

\begin{proof}
The proof is similar to the proof of Theorem~\ref{thm:regularCW}. We prove by induction on $\ell(w) - \ell(v)$. The base case when $\ell(w) - \ell(v) =0$ is trivial. 
		
It follows by induction that   $\overline{ \CP_{K, (v',w'), > 0}}$ is a regular CW complex homeomorphic to a closed ball of dimension $\ell(w') - \ell(v')$ for any $(v',w') \in Q_K$ with $\ell(w') - \ell(v') < \ell(w) - \ell(v)$. Therefore by Proposition~\ref{prop:CPmanifold}, $\partial  \overline{ \CP_{K, (v,w), > 0}}=\bigsqcup_{(s,t) < (v,w) \in Q_K} \CP_{K, (s,t), > 0}$ is a regular CW complex. Its face poset is $\{(v',w') \vert (v',w') < (v,w)\} \subset Q_K$. By adding the maximal element $(v, w)$ and the minimal element $\{\hat{0}\}$, we obtain the new poset $\{(v',w') \vert (v',w') \le  (v,w)\} \bigsqcup \{\hat{0}\} \subset Q_K \bigsqcup \{\hat{0}\}$. 

By \cite[Theorem~4.1]{BH21}, this poset is graded, thin, and shellable. By Theorem~\ref{thm:CW},  $\partial \overline{ \CP_{K, (v,w), > 0}}$  is a regular CW complex homeomorphic to a sphere of dimension $\ell(w) - \ell(v) -1$. Now the theorem follows from Proposition~\ref{prop:CPmanifold} and Theorem~\ref{thm:poincare}.
\end{proof}


\subsection{$J$-total positivity}\label{sec:JTP}

\subsubsection{The totally nonpositive part $\CB_{\le 0}$}\label{sec:negativeg}

Let $\iota: G \rightarrow G$ be the unique group automorphism that is identity on $T$ and maps $x_{i}(a)$ to $x_{i}(-a)$ and $y_{i}(a)$ to $y_{i}(-a)$ for any $i \in I$, $a \in \BR$. Let $U^-_{\le 0}=(U^-_{\ge 0}) \i = \iota(U^-_{\ge 0})$ be the submonoid of $U^-$ generated by $y_i(a)$ for $i \in I$ and $a \in \BR_{<0}$. Since $B^+$ is stable under $\iota$, we denote the induced automorphism on $\CB$ by $\iota$ as well. 

Similar to the definition of $\CB_{\ge 0}$, let $\CB_{\le 0}$ be the closure of $U^-_{\le 0} B^+/B^+$ with respect to the Hausdorff topology. For any $v \le w$, we set $\CB_{v, w, <0}=\CB_{\le 0} \bigcap \OCB_{v, w}$. 
It is clear that we have isomorphisms $\iota: \CB_{\ge 0} \cong \CB_{\le 0}$, and $\iota:  \CB_{v, w, > 0}\cong \CB_{v, w, > 0}$.

We fix a reduced expression $\bf w$. Let $\bf v_+$ be the unique positive subexpression for $v$ in $\bf w$. We define $G_{{\bf v_+}, {\bf w}, <0}$ in the similar way as $G_{{\bf v_+}, {\bf w}, >0}$ in \S\ref{Rp-BH}, but using $y_i(\BR_{<0})$ instead of $y_i(\BR_{>0})$ and $\dot s_i \i$ instead of $\dot s_i$. It is clear that $\iota (G_{{\bf v_+}, {\bf w}, <0}) = G_{{\bf v_+}, {\bf w}, >0}$ and $G_{{\bf v_+}, {\bf w}, <0} B^+/ B^+ = \CB_{v, w, <0}$. 

\subsubsection{$J$-total nonnegative flag varieties}Let $J \subset I$. Let $U^-_{J, \ge 0}$ be the submonoid of $U^-$ generated by $y_i(a)$ for $i \in J$ and $a \in \Rp$. Since $U^-_{\ge 0}=\bigsqcup_{w \in W} U^-_{w, >0}$ and $U^-_{J, \ge 0}=\bigsqcup_{w \in W_J} U^-_{w, >0}$, we have $U^-_{J, \ge 0}=U^-_{\ge 0} \bigcap L_J$. Moreover, let $\pi_J: P^-_J \to L_J$ be the projection map. Then we have $\pi_J(U^-_{\ge 0})=U^-_{J, \ge 0}$ and $\pi_J(U^-_{\le 0}) \i=U^-_{J, \ge 0}$. Set $${}^J U^-_{\succeq 0}=\{h_1 \pi_J(h_2) \i h_2 \vert h_1 \in U^-_{J, \ge 0}, h_2 \in U^-_{\le 0}\}.$$ If $J=I$, then ${}^J U^-_{\succeq 0}=U^-_{\ge 0}$. If $J=\emptyset$, then ${}^J U^-_{\succeq 0}=U^-_{\le 0}$. 

For $v \in W_J$ and $w \in {}^J W$, we set $${}^J U^-_{v, w, >0}=\{h_1 \pi_J(h_2) \i h_2 \vert h_1 \in U^-_{v, >0}, h_2 \in U^-_{w, <0}\}.$$ Then ${}^J U^-_{v, w, >0} \cong U^-_{v, >0} \times U^-_{w, <0} \cong \Rp^{\ell(v)+\ell(w)}$ is a cell and ${}^J U^-_{\succeq 0}=\bigsqcup_{v \in W_J, w \in {}^J W} {}^J U^-_{v, w, >0}$. One may also see that each cell ${}^J U^-_{v, w, >0}$ is locally closed in $U^-$ and thus ${}^J U^-_{\succeq 0}$ is a constructible subset of $U^-$. It is worth pointing out that ${}^J U^-_{\succeq 0}$, in general, is not closed in $U^-$. 

We define the $J$-totally nonnegative flag variety $\JCB_{\ge 0}$ to be the closure of ${}^J U^-_{\succeq 0}B^+ / B^+$ in $\CB$ with respect to the Hausdorff topology. Note that if $v_1 \le v_2$, $w_1 \le w_2$, then ${}^J U^-_{v_1, w_1, >0}$ is contained in the Hausdorff closure of ${}^J U^-_{v_2, w_2, >0}$. Thus 
$$
\JCB_{\ge 0}=\bigcup_{v \in W_J, w \in {}^J W} \overline{{}^J U^-_{v, w, >0}B^+/ B^+}.
$$

For $w_1 \leJ w_2$, let $\JCB_{w_1, w_2, >0}=\OJCB_{w_1, w_2} \bigcap \JCB_{\ge 0}$. We call $\JCB_{w_1, w_2, >0}$ the {\em totally positive $J$-Richardson variety}. Then $\JCB_{\ge 0}=\bigsqcup_{w_1 \leJ w_2} \JCB_{w_1, w_2, >0}$. 

If $J=\emptyset$, then $\JCB_{\ge 0}=\CB_{\le 0}$ and $\JCB_{w_1, w_2, >0}=\CB_{w_1, w_2, <0}$. If $J=I$, then $\JCB_{\ge 0}=\CB_{\ge 0}$ and $\JCB_{w_1, w_2, >0}=\CB_{w_2, w_1, >0}$.

\subsubsection{Main results on $\JCB_{\ge 0}$}  
We collect the main results on $\JCB_{\ge 0}$ in this subsection. The proof relies on the following result, which will be proved in \S\ref{sec:pf1}.

\begin{prop}[$\club$]\label{prop:J}
Let $v \leJ w$ in $W$.  We have 
\begin{enumerate}
\item $\JCB_{v, w, \ge 0 } = \overline{ \JCB_{v, w, > 0 }}= \bigsqcup_{v \le v' \le w' \le w} \JCB_{v', w', >0}$. 
\item $\JCB_{v, w, > 0 } $ is a connected component of $\OJCB_{v, w}(\BR)$.
\item For any $u \in W$ with $v \leJ u \leJ w$, we have $\JCB_{v, w, >0} \subset \dot u U^- \cdot B^+$.
\end{enumerate}
\end{prop}

Combining Proposition~\ref{prop:J} with Theorem~\ref{thm:product}, Theorem~\ref{thm:regularCW} and Proposition~\ref{prop:lkregular}, we have the main result for the $J$-total positivity. 

\begin{thm} \label{thm:J}
Let $v \leJ w$. Then 
\begin{enumerate}
	\item For any $u \in W$ with $v \leJ u \leJ w$, the  map ${}^J c_u$ restricts to an isomorphism 
	\[
	 {}^J c_u: \JCB_{v, w, > 0 }  \cong \JCB_{v, u, >  0 }  \times \JCB_{u, w,> 0 }.
	\]
	\item $\JCB_{v, w, \ge 0 }$ is a regular CW complex homeomorphic to a closed ball of dimension $\Jell(w) - \Jell(v)$.
\end{enumerate}

\end{thm}

\subsection{Links} 

In this subsection, we consider the link of the identity in subspaces of $U^-$. The cases we considered here can be regarded as the special cases of the links of some totally positive (ordinary, $J$-, projected) Richardson varieties in the previous subsections.

 \subsubsection{}
Let $\lambda \in X^{++}$ and $w \in W$. We consider the embedding $U^- \xrightarrow{u \mapsto u \cdot \eta_{\lambda}} {{V}}^\lambda$. We identify $U^-(\BR)$ with its image.  The image of $U^-(\BR) \bigcap B^+ \dot{w} B^+$ is contained in a finite dimensional subspace $L \subset {{V}}^\lambda(\BR)$. We equip $L$ with the standard Euclidean norm. For any $1<w' \le w$, we define 
$$
Lk(U^-_{w', >0})= U^-_{w', >0}  \bigcap \{ || x|| =1 \vert x \in L\}.
$$

Let $Lk( \overline{U^-_{w, >0}} )=  \overline{U^-_{w, >0}} \bigcap \{|| x|| =1 \vert x \in L\}$. Since the $\overline{U^-_{w, >0}}=\bigsqcup_{w' \le w} U^-_{w', >0}$, we have $Lk( \overline{U^-_{w, >0}} )=\bigsqcup_{1<w' \le w} Lk(U^-_{w', >0})$.

\begin{thm} \label{thm:lkregular}
$Lk( \overline{U^-_{w, >0}} )=\bigsqcup_{1<w' \le w}Lk(U^-_{w', >0})$ is a regular CW complex homeomorphic to a closed ball of dimension $\ell(w) -1$.
\end{thm}
\begin{remark}
This proves \cite[Conjecture~10.2 (1)]{GKL}, and generalizes the main result in \cite{Her}.
\end{remark}

\begin{proof}
We have a stratified isomorphism $Lk(\overline{ U^-_{w', >0}}) \cong Lk_{1}(\overline{ \CB_{1,w', >0}})$ as defined in \S\ref{sec:link}. We have verified the assumptions in Theorem~\ref{thm:regularCW} for $Y_{1, w}=\CB_{1, w, >0}$ in Lemma~\ref{lem-in-u} and Proposition~\ref{prop:CB}. Hence the theorem follows by Proposition~\ref{prop:lkregular}.
\end{proof}

\subsubsection{} Let $K \subset I$. We consider the map  
$$
pr_K: U^- \to U^-/U^-_K \cong U_{P^-_K}, \quad g \mapsto g \pi_K(g) \i.
$$ 
Then $pr_K(U^-_{w, >0})=pr_K(U^-_{w^K, >0})$, for any $w \in W$. 
We define 
$$
Lk(\overline{ pr_K(U^-_{w, > 0})})=\overline{pr_K(U^-_{w, >0})} \bigcap \{|| x|| =1 \vert x \in L\}.
$$
We have $Lk(\overline{ pr_K(U^-_{w, > 0})}) \cong Lk_1(\overline{\CP_{K, (1, w^K), > 0}})$, where $Lk_1(\overline{\CP_{K, (1, w^K), > 0}})$ can be defined entirely similar to \S\ref{sec:link} using a singular dominant weight $\lambda$. Therefore, via the compatibility in Proposition~\ref{prop:compatible}, we actually have a stratified isomorphism $Lk(\overline{ pr_K(U^-_{w, > 0})})  \cong Lk_1(\overline{\CP_{K, (1, w^K), > 0}})$. Note that $\overline{pr_K(U^-_{w, >0})} \supsetneqq \bigsqcup_{w' \in W^K, w' \le w} pr_K(U^-_{w', >0})$. The remaining stratum of $\overline{pr_K(U^-_{w, >0})}$ corresponds to $\CP_{K, (v, w'), >0}$ for $v \in W_J$, $w' \in W^K$ with $w' v \le w^K$.

Thanks to Proposition~\ref{prop:lkregular} and Theorem~\ref{thm:J}, we conclude that  

(a) $Lk(\overline{ pr_K(U^-_{w, > 0})})$ (with the stratification arising arising from $Q_K$) is a regular CW complex homeomorphic to a closed ball of dimension $\ell(w)-1$.

\subsubsection{}
Let $J \subset I$ and $\lambda \in X^{++}$. Let $v \in W_J$ and $w \in {}^J W$. By \cite[Theorem~5]{BD}, the image of  $U^-(\BR) \bigcap \JB^- \dot{v} B^+ \bigcap \JB^+ \dot{w} B^+$ under the embedding  $U^- \xrightarrow{u \mapsto u \cdot \eta_{\lambda}} {{V}}^\lambda$ lies in a finite-dimensional subspace $L \subset  {{V}}^\lambda (\BR)$. 

We define 
\[
Lk(\overline{ {}^J U^-_{v, w, >0}})=\overline{{}^J U^-_{v, w, >0}} \bigcap \{|| x|| =1 \vert x \in L\}
\]

We have $Lk(\overline{ {}^J U^-_{v, w, >0}}) \cong Lk_1(\overline{\JCB_{v, w, >0}})$. Note that $\overline{{}^J U^-_{v, w, >0}} \supsetneqq \bigsqcup_{v' \le v, w' \le w} {}^J U^-_{v', w', >0}$. The remaining pieces of $\overline{{}^J U^-_{v', w', >0}}$ corresponds to certain $J$-Richardson varieties arising from the twisted Bruhat order $\leJ$.  

Thanks to Proposition~\ref{prop:lkregular} and Theorem~\ref{thm:J}, we conclude that 

(a) $Lk(\overline{ {}^J U^-_{v, w, >0}})$ (with the stratification arising from $\leJ$) is a regular CW complex homeomorphic to a closed ball of dimension $\ell(vw)-1$.


\section{$J$-totally positivity on $\CB$}\label{sec:6}

\subsection{The main result} The main purpose of this section is to give an explicit description of $\JCB_{u, w, >0}$ in the special case where $w \in {}^J W$. 

Let $ w \in {}^JW$. Note that $u \leJ w$ if and only if ${}^J u \le w$. We fix a reduced expression $\bf w$. Let $\bf {}^J u_+$ be the unique positive subexpression for ${}^J u$ in $\bf w$. Set 
$$
{}^J G_{u, \bf w, >0}=\{h_1 \pi_J(h_2 {}^J \dot u \i) \i h_2 \vert h_1 \in U^-_{u_J, >0}, h_2 \in G_{\bf {}^J u_+, \bf w, <0}\} \cong U^-_{u_J, >0} \times G_{\bf {}^J u_+, \bf w, <0}.
$$

Note that ${}^J G_{u, \bf w, >0} \cdot B^+/B^+$ is connected and ${}^J G_{u, \bf w, >0} \cdot B^+/B^+ \subset \OJCB_{u, w}$.  Now we state the main result of this section. 

\begin{prop}\label{prop:typeII}
Let $w \in {}^J W$ and $u \in W$ with $u \leJ w$. Then 

(1) ${}^J G_{u, \bf w, >0} \cdot B^+/B^+ $ is a connected component of $\OJCB_{u, w}(\BR)$. 

(2) For any $ w' \in W$ with  $u   \leJ w' \leJ w$, we have 
	$\JCB_{u, w', >0} = \overline{\JCB_{u, w, >0}} \bigcap \OJCB_{u, w'}$.
	
(3) Let $\bf w$ be a reduced expression of $w$. Then the following map is an isomorphism: $${}^J G_{u, \bf w, >0}  \to \JCB_{u, w, >0}, \qquad g \mapsto g \cdot B^+/B^+.$$
\end{prop}
This proposition proves a special case of Proposition~\ref{prop:J}. We outline our strategy of the proof.  Part (1) follows from \S \ref{prop:CC}. Part (2) follows from Corollary~\ref{cor:II}. We remark that \S \ref{sec:keylemma} plays a crucial role in the proof. Finally, Part (3) is proved in \S\ref{sec:pf3}.
 

\subsection{Connected components}\label{sec:conn}
Recall that $\JB^+=U^-_J \ltimes T U_{P^+_J}$ and $B^-=U^-_J \ltimes T U_{P^-_J}$. Let $p_{+, J}: \JB^+ \to U^-_J$  be the projection map. 
For any $w \in {}^J W$, we define 
$$
\phi_{w, J}: \JB^+ \dot w \cdot B^+/B^+ \to B^+ \dot w \cdot B^+/B^+, \qquad b \dot w \cdot B^+/B^+ \mapsto p_{+, J}(b) \i b \dot w \cdot B^+/B^+.
$$ 
We have the isomorphism $\JB^+ \dot w \cdot B^+/B^+ \xrightarrow{\sim} U^-_J \times  (B^+ \dot w \cdot B^+/B^+)$. For any $u \in {}^J W$, we define 
$$
\psi^{u, J}: B^- \dot u \cdot B^+/B^+ \to \JB^- \dot u \cdot B^+/B^+, \qquad b \dot u \cdot B^+/B^+ \mapsto \pi_J(b) \i b \dot u \cdot B^+/B^+.
$$

It is easy to see that the maps $\phi_{w, J}$ and $\psi^{u, J}$ are well-defined. In the special case where $u, w \in {}^J W$ with $u \leJ w$, $\phi_{w, J}: \OJCB_{u, w}  \to \OCB_{u, w}$ is inverse to 
$\psi^{u, J}: \OCB_{u, w} \to \OJCB_{u, w}$ and hence we have an isomorphism $\OJCB_{u, w} \cong \OCB_{u, w}$. 

\begin{lem}\label{lem:limit}
Let $h_{1,n}, h_{2,n}  \in U^-_{\ge 0}$ be two sequences such that $\lim_{n \to \infty} h_{1,n} h_{2,n}$ exits. Then we have convergent subsequences  $\{h_{1,n_i}\} $ and $\{h_{2,n_j}\}$. 
\end{lem}
\begin{remark}
Note that all limits are in $ U^-_{\ge 0}$, provided they exist. 
\end{remark}
\begin{proof}
It suffices to prove the statement for $h_{1,n} = y_{i}(a_n)$ for some $i \in I$ and $a_{n} \in \BR_{\ge 0}$. It suffices to prove $\{a_n\}$ is bounded above. Consider the group morphism $r_i: U^-  \rightarrow \BR$ mapping $y_{i}(a)$ to $a$, and $y_j(b)$ to $0$ for $j \neq i$. Then it is clear if $\{a_n\}$ is unbounded, then $r_{i}(h_{1,n} h_{2,n}) \ge a_n$ would diverge.
\end{proof}

\subsubsection{Proof of Proposition \ref{prop:typeII} (1)}\label{prop:CC}

Let $$\pi: U^+_J \dot{u}_J U^-_{P^-_J} {}^J \dot{u} \cdot B^+ / B^+ =\OJCB^u  \rightarrow \OCB_{u_J} = U^+  \dot{u}_J  \cdot B^+ / B^+$$ be the projection map. We show that

(a) ${}^J G_{u, \bf w, >0} \cdot B^+/B^+= \{ z \in \OJCB_{u, w} \vert \phi_{w,J}(z) \in \CB_{{}^Ju, w, <0}, \pi(z) \in \CB_{1,u_J, >0}\}$.

For any $z \in {}^J G_{u, \bf w, >0} \cdot B^+/B^+$, $\phi_{w,J}(z) \in \CB_{{}^Ju, w, <0}$ and $\pi(z) \in \CB_{1,u_J, >0}$. Now let $z \in \OJCB_{u, w}$ with $\phi_{w,J}(z) \in \CB_{{}^Ju, w, <0}$ and $\pi(z) \in \CB_{1,u_J, >0}$. Since $\phi_{w,J}(z) \in \CB_{{}^Ju, w, <0}$, we have $z = hg B^+/B^+$ for some $h \in U^-_J$ and $g  \in G_{{}^Ju, w, <0}$. Then $\pi(z) = h \pi_J(g{}^J\dot{u}^{-1}) B^+/B^+ \in \CB_{1,u_J, >0}$. Since $h \pi_J(g{}^J\dot{u}^{-1}) \in U^-_J$, we have $h \pi_J(g{}^J\dot{u}^{-1})   \in U^-_{u_J, >0}$. This shows that $z \in {}^J G_{u, \bf w, >0}$. (a) is proved. 

We then show that 

(b) ${}^J G_{u, \bf w, >0} \cdot B^+/B^+$ is open in $\OJCB_{u, w}(\BR)$. 

The image of $\OJCB_{u, w}$ under the isomorphism $ \JB^+ \dot w \cdot B^+/B^+ \xrightarrow{\sim}  U^-_J \times  (B^+ \dot w \cdot B^+/B^+)$ is in $U^-_J  \times  (B^+ \dot w \cdot B^+ / B^+ \bigcap U^-_J \cdot \OJCB^u)$, hence in $U^-_J  \times  (B^+ \dot w \cdot B^+ /B^+ \bigcap P^-_J \,\, {}^J\dot{u}\cdot \OCB )$. 
Note that $U^-_J  \times   \OCB_{{}^Ju, w} $ is open in $U^-_J  \times  (B^+ \dot w \cdot B^+ /B^+ \bigcap P^-_J \,\, {}^J\dot{u}\cdot \OCB )$ and  $U^-_J(\BR) \times  \CB_{{}^Ju, w, <0}$ is open in  $U^-_J(\BR)  \times   \OCB_{{}^Ju, w}(\BR) $. Similarly we see that $\CB_{1,u_J, >0}$ is open in $U^+(\BR)  \dot{u}_J  \cdot B^+ / B^+$. (b) is proved. 

Finally we show that 

(c) ${}^J G_{u, \bf w, >0} \cdot B^+/B^+$ is closed in $\OJCB_{u, w}(\BR)$.

Let $z \in \overline{{}^J G_{u, \bf w, >0} \cdot B^+/B^+} \bigcap \OJCB_{u,w}$. We assume $z = \lim_{n \to \infty} h_{1,n} h_{2,n} h_{3,n} B^+/B^+$,
where $h_{1,n} \in U^-_{u_J, >0}, h_{3,n} \in  G_{{}^Ju, w, <0}$ and $h_{2,n}= \pi_J(h_{3,n}{}^J\dot{u}^{-1})^{-1} \in U^-_{J, \ge 0}$.
 
Thanks to the isomorphism $\JB^+ \dot w \cdot B^+/B^+ \xrightarrow{\sim}  U^-_J \times  (B^+ \dot w \cdot B^+/B^+)$,  $\lim_{n \to \infty} h_{1,n} h_{2,n}$  and $\lim_{n \to \infty}   h_{3,n} B^+/B^+ $ exist. Moreover, $\lim_{n \to \infty} h_{3,n} B^+/B^+ \in \CB_{\ge 0} \bigcap \OCB_w \bigcap P^-_J \, {}^J\dot{u} B^+/B^+$. Thus $\lim_{n \to \infty} h_{3,n} B^+/B^+= h B^+/B^+$ for some $h \in G_{v\, {}^Ju, {\bf w}, <0}$ with $v \in W_J$. Thanks to Lemma~\ref{lem:limit}, we can find convergent subsequences of  $  h_{1,n}  $ and  $  h_{2,n}$. Without loss of generality, we can assume both $ \lim_{n \to \infty} h_{1,n}  $ and $ \lim_{n \to \infty} h_{1,n}$ exist and hence converge in $U^-_{J, \ge 0}$.  

Hence $ \lim_{n \to \infty} h_{2,n} h_{3,n} B^+/B^+$ exists and is in $ \overline{ U^-_{P^-_J}{}^J\dot{u} B^+/B^+} \bigcap P^-_J \dot{u} B^+/B^+$. Since ${}^J u \in {}^J W$,  $U^-_{P^-_J} {}^J\dot{u} B^+/B^+$ is closed in $P^-_J \dot{u} B^+/B^+$.  So $\lim_{n \to \infty}   h_{2,n} h_{3,n} B^+/B^+ \in U^-_{P^-_J} {}^J\dot{u} B^+/B^+$. There shows that $v = 1$ and $\lim_{n \to \infty} h_{2,n} = \pi_J(h{}^J\dot{u}^{-1})^{-1}$. Finally note that 
\begin{align*} \pi(z) &= \lim_{n \to \infty}\pi(h_{1,n} h_{2,n} h_{3,n} B^+/B^+) \\ &= \lim_{n \to \infty} h_{1,n} B^+ /B^+ \in U^-_{J,\ge 0} B^+/B^+ \bigcap U^+  \dot{u}_J  \cdot B^+ / B^+.
\end{align*} We have $ \lim_{n \to \infty} h_{1,n} \in U^-_{u_J, >0}$. We conclude that $z \in {}^J G_{u, \bf w, >0} \cdot B^+/B^+$.

\subsection{Inside open subspaces} \label{sec:GKL}

We have the following simple equalities on the product of $x_i$ with $y_i$ for $i \in I$. Let $a, b, c>0$ and $i \neq j$ in $I$. Then
\begin{align}
\notag x_i(a) y_j(\pm b) &= y_j(\pm b) x_i(a),\\
x_i(a) y_i(b+c) &= y_i(\frac{b+c}{a(b+c)+1}) \alpha_i^\vee (a(b+c)+1) x_i(\frac{a}{a(b+c)+1}), \label{eq:ad}\\
\notag x_i(\frac{a}{a(b+c)+1}) y_i(-c) &= y_i(\frac{-c(a(b+c)+1)}{ab+1}) \alpha_i^\vee (\frac{a(b)+1}{a(b+c)+1}) x_i(\frac{a}{ab+1}). 
\end{align}

By direct calculation using the equalities \eqref{eq:ad}, we have for any $i \in I$, $a>0$ and $g_1 \in U^-_{w, <0}$, $x_i(a) \pi_J(g) \i g \in \pi_J(g_2) \i g_2 B^+$ for some $g_2 \in U^-_{w, <0}$. Thus

(a) Let $w \in W$. Then for any $g \in U^-_{w, <0}$ and $b \in B^+_{\ge 0}$,  $b \pi_J(g) \i g \in  \pi_J(g') \i g'  t U^+$ for some $g' \in U^-_{w, <0}$ and $t \in T_{>0}$. 

We also have some results on the product of totally nonnegative part of $U^\pm$ with certain Weyl group elements. 
Let $w, w_1, w_2 \in W$ be such that $w_1 w_2 = w$ and $\ell(w_1) + \ell(w_2) =\ell(w)$ and $h \in U^-_{w, >0}$, $b\in U^+_{w^{-1},>0}$. By \cite[Lemma~5.6]{GKL} and its proof, we have 
\begin{align}
\notag \dot{w}^{-1}  h &\in (U^- \bigcap \dot{w}^{-1} U^+ \dot{w})  U^{+}_{w^{-1}, >0}  T_{>0};\\
 \dot{w_1}^{-1}  h &\in (U^- \bigcap \dot{w_1}^{-1} U^+ \dot{w_1})  U^{-}_{w_2, >0}   U^{+}_{w_1^{-1}, >0}  T_{>0}; \label{eq:GKL}\\
 \notag \dot{w}  b &\in  (U^+ \bigcap \dot{w} U^- \dot{w}^{-1} )  U^{-}_{w, >0}   T_{>0};\\
  \notag \dot{w_2}  b &\in (U^+ \bigcap \dot{w_2} U^- \dot{w_2})  U^{+}_{w_1^{-1}, >0}   U^{-}_{w_2, >0}   T_{>0}. 
\end{align} 

More generally, we have

(b) if $w_1 \le w$, then $ \dot{w_1}^{-1}  h  \in  U^-  U^{+}_{w_1^{-1}, >0}  T_{>0}$.

\begin{lem}\label{lem:typeIinU}
Let $v \in W_J$, $w \in {}^JW$ and $u \in W$ with $v\leJ u \leJ w$. Then we have 
\[
{}^J U^-_{v, w, >0} \subset \dot{u} U^- B^+.
\]
\end{lem}

\begin{proof}
Since $v\leJ u \leJ w$, we have $u_J \le v$ and ${}^J u \le w$. Let $h  \in U^-_{v, >0}$ and $g \in U^-_{w, <0}$.  We have $\dot{u}_J^{-1} h = h_1 b_1 t_1$ for some $h_1 \in U^-_J$, $b_1 \in U^+_{J, \ge 0}$ and $t_1 \in T_{>0}$. By \S\ref{sec:GKL} (a), $b_1 t_1 \pi_J(g) \i g \in \pi_J(g_1) \i g_1 B^+$ for some $g_1 \in U^-_{w, <0}$. 

Via a variation of \S\ref{sec:GKL} (b) (for $g_1 \in U^-_{\le 0}$), we have
${}^J\dot{u}^{-1} g_1 \in U^- B^+$. 
Thus 
\begin{align*}
	\dot{u}^{-1} h \pi_J(g) \i g &=   {}^J\dot{u}^{-1} h_1  b_1 t_1 \pi_J(g) \i g  \in {}^J\dot{u}^{-1}  \pi_J(g_1) \i g_1  B^+ \subset U^-  B^+.
\end{align*}
This finishes the proof.
\end{proof}

\begin{lem}\label{lem:typeII}
Let $u \in W$ and $w \in {}^JW$ with $u \leJ w$. Let $v \in W_J$ with $u_J \le v $. Then 
\[
	 \Jc_{u,-} ({}^J U^-_{v, w, >0} \cdot B^+/B^+) = \JCB_{u,w, >0}.
\]

In particular, ${}^J U^-_{v, w, >0} \cdot B^+/B^+ = \JCB_{v,w, >0}$.
\end{lem}
\begin{proof}
For $x \in W_J$ and $y \in {}^J W$ with $x \leJ y$, we simply write $\OJCB'_{x, y, >0}$ for ${}^J U^-_{x, y, >0} \cdot B^+/B^+$. Recall that $\JCB_{u,w, >0} = \bigcup_{v'\in W_J, w' \in {}^JW} (\overline{\OJCB'_{v', w', >0}} \bigcap \OJCB_{u,w})$. By Lemma~\ref{lem:key2} and Lemma \ref{lem:typeIinU} we have  
\[
\Jc_{u,-} (\OJCB'_{v, w, >0}) = \overline{\OJCB'_{v, w, >0}} \bigcap \OJCB_{u,w} \subset  \JCB_{u,w, >0}. 
\]
In particular, $\Jc_{u,-} (\OJCB'_{v, w, >0}) $ is a connected component of $\OJCB_{u,w}(\BR)$. 
Note that for any $v'\in W_J, w' \in {}^JW$, we can always find $v''\in W_J, w'' \in {}^JW$ such that  $v' \le v''$, $ w' \le w''$, $u_J \le v'' $,  $u \leJ w''$. In particular, $\OJCB'_{v', w', >0} \subset \overline{\OJCB'_{v'', w'', >0}}$.  Without loss of generality, we can further assume $v\le v''$ and $ w \le w''$.
Let us fix such $v''$ and $w''$. Then it remains to the show:
\begin{itemize}	
	\item[(a)] $\overline{\OJCB'_{v'', w'', >0}} \bigcap \OJCB_{u,w} = \overline{\OJCB'_{v'', w, >0}} \bigcap \OJCB_{u,w}$;
	\item[(b)] $\overline{\OJCB'_{v'', w, >0}} \bigcap \OJCB_{u,w} = \overline{\OJCB'_{v, w, >0}} \bigcap \OJCB_{u,w}$.
\end{itemize}
Note that $\OJCB'_{v'', w'', >0} \subset \dot{w}U^-B^+/B^+$ and $\OJCB'_{v'', w, >0} \subset \dot{u}U^-B^+/B^+$. By Lemma~\ref{lem:key2} and Lemma \ref{lem:typeIinU}, we have $\overline{\OJCB'_{v'', w'', >0}} \bigcap \OJCB_{v'',w} = \Jc_{w,+}(\OJCB'_{v'', w'', >0})$. Since $w \le w''$, we have $\OJCB'_{v'', w, >0} \subset \overline{\OJCB'_{v'', w'', >0}} \bigcap \OJCB_{v'',w}$. Since both sides are connected components of $\OJCB_{v'', w}(\BR)$, we must have $\OJCB'_{v'', w, >0} = \Jc_{w,+}(\OJCB'_{v'', w'', >0}) =\overline{\OJCB'_{v'', w'', >0}} \bigcap \OJCB_{v'',w}$. Now (a) follows by Lemma~\ref{lem:key}.

We then obtain $\overline{\OJCB'_{v'', w'', >0}} \bigcap \OJCB_{u,w}= \overline{\OJCB'_{v'', w, >0}}  \bigcap \OJCB_{u,w}=  \Jc_{u,-} (\OJCB'_{v'', w, >0})$. Since $v \le v''$, we obtain that $\overline{\OJCB'_{v, w, >0}}  \bigcap \OJCB_{u,w} \subset \overline{\OJCB'_{v'', w, >0}}  \bigcap \OJCB_{u,w}$. Since both sides are connected components of $\OJCB_{u, w}(\BR)$, we have 
\[
	 \overline{\OJCB'_{v'', w'', >0}} \bigcap \OJCB_{u,w}  =  \Jc_{u,-} (\OJCB'_{v'', w, >0}) = \Jc_{u,-} (\OJCB'_{v, w, >0}).
\]
This finishes the proof. 
\end{proof}

Combining Lemma \ref{lem:typeII} with Lemma~\ref{lem:key} \& \ref{lem:key2}, we have the following consequences.

\begin{cor}\label{cor:II} Let $u \in W$ and $w \in {}^JW$ with $u \leJ w$. Then

(1) $\JCB_{u,w, >0}=\overline{\JCB_{v,w, >0}} \bigcap \OJCB_{u,w}$ is a connected component in $ \OJCB_{u,w}(\BR)$.

(2) For any $w' \in W$ with $u \leJ w' \leJ w$, $\JCB_{u,w', >0}=\overline{\JCB_{u,w, >0}} \bigcap \OJCB_{u,w'}$.
\end{cor}

We remark that Proposition~\ref{prop:typeII} (2) follows now.

 \subsection{Positivity} \label{sec:typeII2}

\begin{lem}\label{lem:Jrw}
Let $u, w \in {}^J W$ with $u \leJ w$. Then the isomorphism $\phi_{w, J}: \OJCB_{u, w} \cong \OCB_{u, w}$ restricts to an isomorphism $\JCB_{u, w, >0}  \cong \CB_{u,w, <0}$. Moreover, ${}^J G_{u, \bf w, >0} \cdot B^+/B^+ =  \JCB_{u, w, >0}$.
\end{lem}

\begin{proof}We first prove the statement for $ u = 1$, that is, $  \JCB_{1, w, >0}  \cong \CB_{1,w, <0}$. 
 Let $ h \in U^-_{1, w, <0}$. Then $h B^+/B^+ = g \dot{w} B^+/B^+$ for some $g \in U^+$. Since $w \in {}^JW$, we may further assume that $g \in U_{P^+_J}$. By definition, $\pi_J(h)^{-1}  g \in \JB^+$ and $p_{+, J}(\pi_J(h)^{-1}  g)=\pi_J(h)^{-1}$. Thus
\begin{align*}
\phi_{w, J} (\pi_J(h)^{-1} h B^+/B^+) &= \phi_{w, J} (\pi_J(h)^{-1}  g \dot{w} B^+/B^+) =g \dot w B^+/B^+=h B^+/B^+  \in \CB_{1,w, <0}.
\end{align*}
On the other hand, $\psi^{1, J} (hB^+/B^+) = \pi_J(h)^{-1} h B^+/B^+ \in {}^J U^-_{1, w, >0}  \cdot B^+/B^+$. This finishes the proof.

Now we consider the general case. Since $\CB_{u,w, >0}  \subset \overline{\CB_{1,w, <0}}$ and $\JCB_{u,w, >0}  \subset \overline{\JCB_{1,w, <0}}$, we have $  \phi_{w, J}^{-1} (\CB_{u,w, <0}) \subset \JCB_{u,w, >0}$. Thanks to Corollary~\ref{cor:II}, $\JCB_{u,w, >0}$ is a connected component of $\OJCB_{u,w}(\BR)$. Thanks to Theorem~\ref{thm:CB}, $\CB_{u,w, >0} $ is a connected component of $ \OCB_{u,w}(\BR)$. Since isomorphism sends connected components to connected components, we must have $\phi_{w, J}^{-1} (\CB_{u,w, <0}) = \JCB_{u,w, >0}$.
\end{proof}

\subsection{Proof of Propositon \ref{prop:typeII} (3)}\label{sec:pf3}

By Lemma~\ref{lem:Jrw}, ${}^J G_{{}^J u, \bf w, >0} \cdot B^+/B^+ =  \JCB_{{}^J u, w, >0}$. Since $ \JCB_{\ge 0}$ is stable under the action of $U^-_{J, \ge 0}$, we have
\[
	 {}^J G_{u, \bf w, >0} \cdot B^+/B^+ = U^-_{u_J, > 0} {}^J G_{{}^Ju, \bf w, >0} \cdot B^+/B^+\subset  \JCB_{\ge 0} \bigcap \JCB_{u, w}={}^J  \CB_{u, w, >0}.
\]
Both sides are connected components of $\OJCB_{u, w} (\BR)$ by Proposition~\ref{prop:typeII} (1) and Corollary~\ref{cor:II}. Thus ${}^J G_{u, \bf w, >0} \cdot B^+/B^+ =  \JCB_{u, w, >0}$.

On the other hand, ${}^J G_{u, \bf w, >0} \cdot B^+/B^+ \subset U^-_J B^+ \dot w B^+/B^+$. Since $w \in {}^J W$, we have an isomorprhism $i: U^-_J B^+ \dot w B^+/B^+ \cong U^-_J \times B^+ \dot w B^+/B^+$. By definition, for $h_1 \in U^-_{u_J, >0}$ and $h_2 \in G_{{}^J u, \bf w, <0}$, $$i(h_1 \pi_J(h_2 {}^J \dot u \i) \i h_2 \cdot B^+/B^+)=(h_1 \pi_J(h_2 {}^J \dot u \i), h_2 \cdot B^+/B^+).$$ Thus we may recover $h_1$ and $h_2$ from $h_1 \pi_J(h_2 {}^J \dot u \i) \i h_2 \cdot B^+/B^+$. This finishes the proof. 
 


\section{Basic $J$-Richardson varieties}

\subsection{Motivation}
The totally nonnegative partial flag variety $\CP_{K, \ge 0}$ has a cellular decomposition and each cell admits an explicit parametrization. However, $\CP_{K, \ge 0}$ does not have an obvious product structure. On the other hand, for the $J$-total positivity, it is natural to expect that the map ${}^J c_u$ gives the product structure on the $J$-total positivity (this is what we will eventually prove). However, except for the special cases studied in \S\ref{sec:6}, it is rather difficult to understand the general stratum $\JCB_{v, w, >0}$ (the parametrization,  connected components, etc.) 

In this section, we will introduce the basic $J$-Richardson varieties. This family of special $J$-Richardson varieties serves as models for both the projected Richardson varieties and the $J$-Richardson varieties. We will establish such connections in this section. Finally, in the last section, we will show that the totally positive part of the basic $J$-Richardson varieties are compatible with both the totally positive projective Richardson varieties and the general total positive $J$-Richardson varieties. Such compatibility will allow us to bring together the information we have obtained on the totally positive projective Richardson varieties and the general total positive $J$-Richardson varieties and finish the proof of our main results. 

\subsection{The larger Kac-Moody group}\label{sec:tildeG}
Let $G'$ be a Kac-Moody group  associated to the Kac-Moody root datum $(I', A', X', Y', \ldots)$. Let $K' \subset I'$. Following \cite[\S 3]{BH21}, we associate a new Kac-Moody group $\tilde{G'}$ of adjoint type to $G'$. The Dynkin diagram of $\tilde{G'}$ is obtained by glueing two copies of the Dynkin diagram of $G'$ along the subdiagram $K'$.  

We denote by $\tilde{I'}$ the set of simple roots of $\tilde{G'}$. We denote by $(I')^\flat$ and $(I')^\sharp$ the two copies of $I'$. The elements in $(I')^\flat$ (resp. $(I')^\sharp$) are denoted by $i^\flat$ (resp. $i^\sharp$) for $i \in I'$. Then $\tilde I'=(I')^\flat \bigcup (I')^\sharp$ with $(I')^\flat \bigcap (I')^\sharp=\{k^\flat=k^\sharp \vert k \in K'\}$. 

Let $W'$ be the Weyl group of $G'$ and $\tilde W'$ be the Weyl group of $\tilde G'$. We have natural identifications $W' \to \tilde W'_{(I')^\flat}, w \mapsto w^\flat$ and $W' \to \tilde W'_{(I')^\sharp}, w \mapsto w^\sharp$. For $w \in W'_{K'}$, $w^\flat=w^\sharp$. Similarly, we have natural maps $G' \to \tilde L_{(I')^\flat}, g \mapsto g^\flat$ and $G' \to \tilde L_{(I')^\sharp}, g \mapsto g^\sharp$. For $g \in L_{K'}$, $g^\flat=g^\sharp$. 

Let $\tilde \CB'$ be the flag variety of $\tilde G'$. Let $Q_{K'}=\{(v, w) \in W' \times (W')^{K'} \vert v \le w\}$. Define
$$
\tilde \nu: Q_{K'} \to \tilde W', \qquad (v, w) \mapsto (w)^\flat (v \i)^\sharp.
$$
By \cite[Proposition 4.2 (1)]{BH21}, $\tilde \nu$ is compatible with the partial order $\preceq$ on $Q_{K'}$ and the partial order ${}^{(I')^\flat}\!\!\le$ on $\tilde W'$.

\begin{defi}
A $J$-Richardson variety $\OJCB_{v,w}$ is called {\it basic} (with respect to $G'$) if it is of the form ${}^{{I'}^\flat}\!\!\mathring{\tilde \CB}'_{\tilde \nu(\alpha), \tilde \nu(\beta)}$ for some triple $(K', \a, \b)$, where $K'$ is a subset of the simple roots in $G'$ and $\a \preceq \b$ in $Q_{K'}$.
\end{defi}

\subsection{The Birkhoff-Bruhat atlas on $\CP_K$}\label{sec:BBatlas}
The technical definition of basic $J$-Richardson varieties arises from the Birkhoff-Bruhat atlas introduced in \cite{BH21}, which relates the projected Richardson varieties for $G$ with the $J$-Richardson varieties for $\tilde G$.

Let $r \in W^K$. The isomorphism $\s_r: \dot r U^-  \dot r \i \to (\dot r U^-  \dot r \i \bigcap U^+) \times (\dot r U^- \dot r\i \bigcap U^-)$ in \eqref{eq:Jsigma} is compatible with Levi decompositions. The restriction of $\sigma_r$ gives the isomorphism $\dot r U_{P^-_K}  \dot r \i \to (\dot r U_{P^-_K}  \dot  r \i \bigcap U^+) \times (\dot r U_{P^-_K}  \dot r\i \bigcap U^-)$. We define 
$$f_r: \dot r U^- P^+_K \to \tilde G, \quad g \dot r p \mapsto (\s_{r, +} (g) \dot{ r })^{\flat} (g \dot{ r })^{\sharp, -1} \text{ for }g \in \dot r U_{P^-_K} \dot r \i, p \in P^+_K.$$ The map $f_r$ factors through $\dot r B^- P^+_K/P^+_K \subset \CP_{K}$ and induces a morphism $$\tilde c_r: \dot r B^- P^+_K/P^+_K \to \tilde \CB.$$



Let $(v,w) \in Q_K$ and $r \in W^K$ with $(r, r) \preceq (v, w)$ in $Q_K$. By \cite[Theorem 3.2]{BH21}\footnote{In loc.cit, we use $(\s_{r, +} (g) \dot{ r })^{\flat} (\s_{r, -} (g) \dot{ r })^{\sharp, -1}$ instead. However, it differs from $f(g \dot r)$ by multiplying an element in $(U^+)^\sharp$ on the right, the induced maps  to $\tilde \CB$ coincide.} we have the following commutative diagram
\[
\xymatrix{
\mathring{\CP}_{K, (v ,w)} \bigcap \dot r U^- P^+_K/P^+_K \ar[r]^-{\tilde{c}_r}_-{\cong} \ar@{^{(}->}[d] & {}^{I^\flat}\!\! \mathring{\tilde{\CB}}_{\tilde{\nu}(r,r), \tilde{\nu}(v,w)} \ar@{^{(}->}[d] \\
\dot r B^- P^+_K/P^+_K \ar@{^{(}->}[r]^-{\tilde{c}_r} & {}^{I^\flat}\!\! \tilde \CB.
}
\]

In other words, $\tilde{c}_r$ gives an isomorphism from the stratified space $$\dot r B^- P^+_K/P^+_K=\bigsqcup_{(v, w) \in Q_K; (r, r) \preceq (v, w)} \Big(\mathring{\CP}_{K, (v ,w)} \bigcap \dot r U^- P^+_K/P^+_K\Big)$$ into its image in ${}^{I^\flat}\!\! \tilde \CB$ stratified by the ${I^\flat}$-Richardson varieties. 


\subsection{Basic $J$-Richardson varieties in $\CB^\spade$} \label{sec:thick}
Our next goal is to show any $J$-Richardson variety for the Kac-Moody group $G$ can be realized as a basic one with respect to a different Kac-Moody group. 


Set $I^{!}=I \bigsqcup \{0\}$. The generalized Cartan matrix $A^{!}=(a^!_{i, j})_{i, j \in I^{!}}$ is defined by 
\begin{itemize}
	\item for $i, j \in I$, $a^!_{i, j}=a_{i, j}$; 
	
	\item for $i \in I$, $a^!_{i, 0}=a^!_{0, i}=-2$; 
	
	\item $a^!_{0, 0}=2$.  
\end{itemize}

Let $W^{!}$ be the Weyl group associated to $(I^{!}, A^{!})$. Then we have the natural identification $W=W^{!}_I$. Moreover, for any $i \in I$, $s_0 s_i$ is of infinite order in $W^{!}$ and $s_0 x \in {}^I (W^{!})$ for all $x \in W$. 
Now for the triple $(I^!, A^!, J)$, we construct a triple $(\widetilde{I^!}, \widetilde{\,A^!\,}, I^{!^\flat})$ following the construction in \S\ref{sec:tildeG}. We write $I^\spade=\widetilde{I^!}$ and $A^\spade=\widetilde{\,A^!\,}$. Let $G^\spade$ be the minimal Kac-Moody group of adjoint type associated to $(I^\spade, A^\spade)$. Let $W^\spade$ be the Weyl group associated to $G^\spade$ and $\CB^{\spade}$ be the flag variety of $G^{\spade}$. 

\begin{prop}\label{prop:thick}
Let $x \in W$. For any $g \in G$, define $i^\spade_x(g B^+)=g^\sharp (\dot s_0 \dot x)^\sharp (B^\spade)^+$. Then for any $v \leJ w$, we have the following commutative diagram
\[
\begin{tikzcd}
\OJCB_{v, w} \ar[r, "\cong" below, "i^\spade_x" above] \ar[d, hook] & {}^{{I^!}^\flat}\!\! \OCB^\spade_{v^\sharp (s_0 x)^\sharp, w^\sharp (s_0 x)^\sharp}  \ar[d,  hook]   \\
 \CB \ar[r, "i^\spade_x",  hook] & \CB^\spade.
\end{tikzcd}
\] 
\end{prop}

\begin{remark}\label{rem:thick}
(1) In other words, $i^\spade_x$ gives an isomorphism from the stratified space $\CB=\bigsqcup_{v \leJ w} \OJCB_{v, w}$ into its image in $\CB^\spade$ stratified by the ${I^!}^\flat$-Richardson varieties.

(2) Note that $v^\sharp (s_0 x)^\sharp=v_J^\flat ({}^J v \, s_0 x)^\sharp$ and $({}^J v \, s_0 x)^\sharp \in {}^{{I^!}^\flat} \!  W^\spade$. Thus by definition, if $v_J, w_J \le x \i$, then ${}^{{I^!}^\flat} \!  \OCB^\spade_{v^\sharp (s_0 x)^\sharp, w^\sharp (s_0 x)^\sharp}$ is a basic $J$-Richardson variety. 
\end{remark}

\begin{proof} 
Set $\diamondsuit=\spadesuit-\{0^\sharp\}$. Let $\CB^{\diamondsuit}$ be the flag variety of $L^\spade_{\diamondsuit}$. We have 
\begin{align*} {}^{{I^!}^\flat} \!\!  \OCB^\spade_{v^\sharp (s_0 x)^\sharp, w^\sharp (s_0 x)^\sharp} & \subset (P^\spade_{\diamondsuit})^- (\dot s_0 \dot x)^\sharp (B^\spade)^+/(B^\spade)^+ \bigcap (P^\spade_{\diamondsuit})^+ (\dot s_0 \dot x)^\sharp (B^\spade)^+/(B^\spade)^+ \\ &=L^\spade_{\diamondsuit} (\dot s_0 \dot x)^\sharp (B^\spade)^+/(B^\spade)^+.
\end{align*}

The map $g \mapsto g (\dot s_0 \dot x)^\sharp  (B^\spade)^+$ for $g \in L^\spade_{\diamondsuit}$ induces the following Cartesian diagram 
\[
\xymatrix{
{}^{{I^!}^\flat}\!\!  \OCB^{\diamondsuit}_{v^\sharp, w^\sharp} \ar[d] \ar[r]^-{\cong} & {}^{{I^!}^\flat}  \!\!  \OCB^\spade_{v^\sharp (s_0 x)^\sharp, w^\sharp (s_0 x)^\sharp} \ar[d] \\ 
L^\spade_{\diamondsuit}/(L^\spade_{\diamondsuit} \bigcap (B^\spade)^+) \ar[r]^-{\cong} &  L^\spade_{\diamondsuit} (\dot s_0 \dot x)^\sharp (B^\spade)^+/(B^\spade)^+.
}
\]

Moreover ${}^{{I^!}^\flat} \!\! \OCB^{\diamondsuit}_{v^\sharp, w^\sharp} \subset {}^{{I^!}^\flat} \!\!(B^{\diamondsuit})^- (P^{\diamondsuit}_{I^\sharp})^+/(B^\diamondsuit)^+ \bigcap {}^{{I^!}^\flat} \!\!(B^{\diamondsuit})^+ (P^{\diamondsuit}_{I^\sharp})^+/(B^\diamondsuit)^+=(P^{\diamondsuit}_{I^\sharp})^+/(B^\diamondsuit)^+$. The isomoprhism $\CB \to (P^{\diamondsuit}_{I^\sharp})^+/(B^\diamondsuit)^+$, $g B^+ \mapsto g^\sharp (B^\diamondsuit)^+$ induces an isomorphism $\OJCB_{v, w} \cong {}^{{I^!}^\flat} \!\!  \OCB^{\diamondsuit}_{v^\sharp, w^\sharp}$. The proposition is proved. 
\end{proof}

\subsection{Some consequences on the $J$-Richardson varieties}
We combine the results on the projected Richardson varieties and the $J$-Richardson varieties to prove the following result. 

\begin{prop}\label{prop:Jcl}
Let $v \leJ w$. Then 

(1) the $J$-Richardson variety $\OJCB_{v, w} $ is irreducible of dimension $\Jell(w) - \Jell(v)$.

(2) the Zariski closure $\JCB_{v, w}$ of $\OJCB_{v, w}$ equals $\JCB_w \bigcap \JCB^v=\bigsqcup_{v \leJ v' \leJ w' \leJ w} \OJCB_{v', w'}$.
\end{prop}

\begin{remark}
	In the case of reductive groups,  one may deduce both statements for (ordinary) Richardson varieties easily from the transversal intersections of  $B^+$-orbits and  $B^-$-orbits on $\CB$. In the case of Kac-Moody groups, both statements for (ordinary) Richardson varieties are established recently in \cite{Kum1}. Our proof for $J$-Richardson varieties is based on \cite{Kum1}. 
\end{remark} 

\begin{proof}
(1) Set $Q_J^!=\{(a, b) \in W^! \times (W^!)^J \vert a \le b\}$. Let $\nu^!: Q^{!}_J \rightarrow W^\spade$ be the map sending $(a, b)$ to $a^\flat (b^{-1})^\sharp$. Let $x \in W$ with $v_J, w_J \le x \i$. Then we have $$(1, 1) \preceq (v_J, x \i s_0 ({}^J v) \i) \preceq (w_J, x \i s_0 ({}^J w) \i).$$ Moreover, $\nu^!(v_J, x \i s_0 ({}^J v) \i)=v^\sharp (s_0 x)^\sharp$ and $\nu^!(w_J, x \i s_0 ({}^J w) \i)=w^\sharp (s_0 x)^\sharp$. Thus 
\begin{equation}\label{eq:a}
    1 \, {}^{(I^!)^\flat}\!\!\! \le v^\sharp (s_0 x)^\sharp \, {}^{(I^!)^\flat} \!\!\! \le w^\sharp (s_0 x)^\sharp.
\end{equation}

By \cite[Proposition~6.6]{Kum1}, $\mathring{\CB^!}_{w_J, x \i s_0 ({}^J w) \i}$ is irreducible of dimension $$\ell(x \i s_0 ({}^J w) \i)-\ell(w_J)=\ell(x)+1+\Jell(w).$$ 

We have $\mathring{\CP^!}_{J, (w_J, x \i s_0 ({}^J w) \i)} \cong \mathring{\CB^!}_{w_J, x \i s_0 ({}^J w) \i}$. Since $(1, 1) \preceq (w_J, x \i s_0 ({}^J w) \i)$, we have $\mathring{\CP^!}_{J, (w_J, x \i s_0 ({}^J w) \i)} \bigcap (U^!)^- (P^!_J)^+/(P^!_J)^+ \neq \emptyset$. By \S\ref{sec:BBatlas}, 
$$
{}^{{I^!}^\flat}\!\! \OCB^\spade_{1, w^\sharp (s_0 x)^\sharp} \cong \mathring{\CP^!}_{J, (w_J, x \i s_0 ({}^J w) \i)} \bigcap (U^!)^- (P^!_J)^+/(P^!_J)^+
$$
 is also irreducible of dimension $\ell(x)+1+\Jell(w)$. Similarly, ${}^{{I^!}^\flat} \!\!  \OCB^\spade_{1, v^\sharp (s_0 x)^\sharp}$ is irreducible of dimension $\ell(x)+1+\Jell(v)$.

By \eqref{eq:a} and \S\ref{sec:Jcr} (b), ${}^{{I^!}^\flat}\!\!  \OCB^\spade_{1, w^\sharp (s_0 x)^\sharp} \bigcap \dot v^\sharp (\dot s_0 \dot x)^\sharp (U^\spade)^- (B^\spade)^+/(B^\spade)^+ \neq \emptyset$ and we have 
$$
{}^{{I^!}^\flat} \!\! \OCB^\spade_{1, w^\sharp (s_0 x)^\sharp} \bigcap \dot v^\sharp (\dot s_0 \dot x)^\sharp (U^\spade)^- (B^\spade)^+/(B^\spade)^+ \cong {}^{{I^!}^\flat}\!\!  \OCB^\spade_{1, v^\sharp (s_0 x)^\sharp} \times {}^{{I^!}^\flat} \!\! \OCB^\spade_{v^\sharp (s_0 x)^\sharp, w^\sharp (s_0 x)^\sharp}.
$$ 
Since ${}^{{I^!}^\flat} \!\! \OCB^\spade_{1, v^\sharp (s_0 x)^\sharp}$ is irreducible of dimension $\ell(x)+1+\Jell(v)$, we have ${}^{{I^!}^\flat} \!\! \OCB^\spade_{v^\sharp (s_0 x)^\sharp, w^\sharp (s_0 x)^\sharp}$ is irreducible of dimension $\Jell(w)-\Jell(v)$. Now part (1) follows from Proposition \ref{prop:thick}. 

(2) We have $\JCB_{v, w} \subset \JCB_w \bigcap \JCB^v$. By \cite[Theorem 4]{BD}, $$\JCB_w \bigcap \JCB^v=\bigsqcup_{v \leJ v' \leJ w' \leJ w} \OJCB_{v', w'}.$$ 

Let $u \in W$ with $v \leJ u \leJ w$. Set $Y=\OJCB_{v,w}  \bigcap (\dot{u} B^- B^+ /B^+)$. By \S\ref{sec:Jcr} (a), $Y \neq \emptyset$. Let $z \in Y$. Then $T \cdot z \subset \OJCB_{v,w}  \bigcap (\dot{u} B^- B^+ /B^+)$. By the proof of Lemma \ref{lem:key2}, the closure of $T \cdot z$ contains $\dot u B^+/B^+$. Therefore, $\dot u B^+/B^+$ is contained in the Zariski closure of $Y$. We may apply Lemma \ref{lem:key} and the Remark \ref{remark:Zariski} to the Zariski closure of $Y$. By \eqref{eq:atlasJ1} and the remark \ref{remark:Zariski}, $\OJCB_{v, u}={}^J c_{u, +}(Y)$ and $\OJCB_{u, w}={}^J c_{u, -}(Y)$ are contained in the Zariski closure of $Y$, and hence in $\JCB_{v, w}$. 

Now let $v', w' \in W$ with $v \leJ v' \leJ w' \leJ w$. If $v' \neq v$, then we have $\OJCB_{v', w} \subset \JCB_{u, w}$ and $\OJCB_{v', w'} \subset \JCB_{v', w}$. So $\OJCB_{v', w'} \subset \JCB_{u, w}$. If $v'=v$, then we have $\OJCB_{v, w'} \subset \JCB_{v, w}$. Part (2) is proved. 
\end{proof}

\begin{cor}\label{cor:dense}
Let $v \leJ w$. Then $\JCB_{v, w, >0}$ is Zariski dense in $\OJCB_{v, w}$.
\end{cor}

\begin{proof}We denote by $dim_{\BR}(\cdot)$ the $\BR$-dimension of a real semi-algebraic variety (see \cite[\S2.8]{BCR}). We remark that all spaces considered here are semi-algebraic.

By Proposition \ref{prop:typeII} (3), we have an semi-algebraic homeomorphism $\JCB_{v, {}^J w, >0} \cong \Rp^{\Jell({}^J w)-\Jell(v)}$. Therefore the Zariski closure of $\JCB_{v, {}^J w, >0}$ in $\OJCB_{v, {}^J w}$ is irreducible and of dimension $\Jell({}^J w)-\Jell(v)$. Thus by Proposition \ref{prop:Jcl} (1), $\JCB_{v, {}^J w, >0}$ is Zariski dense in $\OJCB_{v, {}^J w}$. 

By definition, $v \leJ w \leJ {}^J w$. By \S\ref{sec:Jcr} (a), $\JCB_{v, {}^J w, >0} \bigcap \dot w U^- B^+/B^+ \neq \emptyset$. 

Set $X=\JCB_{v, {}^J w, >0} \bigcap \dot w U^- B^+/B^+$. By \eqref{eq:atlasJ}, we have $X \cong {}^J c_{w, +}(X) \times {}^J c_{w, -}(X)$. By Lemma \ref{lem:key}, ${}^J c_{w, +}(X)=\overline{X} \bigcap \OJCB_{v, w} \subset \JCB_{v, w, >0}$ and ${}^J c_{w, -}(X)=\overline{X} \bigcap \OJCB_{w, {}^J w} \subset \JCB_{w, {}^J w, >0}$. In particular, $\dim_{\BR}({}^J c_{w, -}(X)) \le \Jell({}^J w)-\Jell(w)$. So $$\dim_{\BR}(\JCB_{v, w, >0}) \ge \dim_{\BR}({}^J c_{w, +}(X))= \dim_{\BR}(X)-\dim_{\BR}({}^J c_{w, -}(X))) \ge \Jell(w)-\Jell(v).$$

By Proposition \ref{prop:Jcl} (1), $\dim_{\BR}(\JCB_{v, w, >0})=\dim_{\BR}({}^J c_{w, +}(X))=\Jell(w)-\Jell(v)$ and $\JCB_{v, w, >0}$ is Zariski dense in $\OJCB_{v, w}$. 
\end{proof}
 

\section{Final part of the proofs}\label{sec:8}

\subsection{Compatibility of total positivities} 

In this subsection we show the compatibility among the total positivity on the projected Richardson varieties, the $J$-total positivity on the basic $J$-Richardson varieties, and the $J$-total positivity on arbitrary $J$-Richardson varieties.  

\begin{prop}\label{prop:compatible}
Let $(v, w) \in Q_K$ and $r \in W^K$ with $(r, r) \preceq (v, w)$ in $Q_K$. Then the isomoprhism $\tilde{c}_r: \mathring{\CP}_{K, (v ,w)} \bigcap \dot r U^- P^+_K/P^+_K \cong  {}^{I^\flat}\!\! \mathring{\tilde{\CB}}_{\tilde{\nu}(r,r), \tilde{\nu}(v,w)}$ restricts to an isomorphism
\[
\tilde{c}_r : \CP_{K, (v ,w), >0} \rightarrow  {}^{I^\flat}\!\! {\tilde{\CB}}_{\tilde{\nu}(r,r), \tilde{\nu}(v,w), >0}.
\]
\end{prop}
\begin{proof}
We claim it suffices to prove the statement for $(1, w')$ with sufficiently large $w'$. 

Note that for any $(v, w) \in Q_K$, there exists $(1, w')$ with $(v, w) \preceq (1, w')$. By Proposition \ref{prop:closure-p} (2), $\overline{\CP_{K, (1 ,w'), >0}} \bigcap \mathring{\CP}_{K, (v, w)}=\CP_{K, (v, w), >0}$. By Proposition \ref{prop:typeII} (2), $$\overline{{}^{I^\flat}\!\! {\tilde{\CB}}_{\tilde{\nu}(r,r), \tilde{\nu}(1,w'), >0}} \bigcap {}^{I^\flat}\!\! \mathring{\tilde{\CB}}_{\tilde{\nu}(r,r), \tilde{\nu}(v,w)}={}^{I^\flat}\!\! {\tilde{\CB}}_{\tilde{\nu}(r,r), \tilde{\nu}(v,w), >0}.$$ Suppose that the statement holds for $(1, w')$. Then the statement for $(v, w)$ follows from the following commutative diagram.
\[
\xymatrix{
\CP_{K, (1 ,w'), >0} \ar[r] \ar[d]^-{\cong} & \overline{\CP_{K, (1 ,w'), >0}} \ar[d]^-{\cong} & \overline{\CP_{K, (1 ,w'), >0}} \bigcap \mathring{\CP}_{K, (v, w)} \ar[l] \ar[d]^-{\cong} \\
{}^{I^\flat}\!\! {\tilde{\CB}}_{\tilde{\nu}(r,r), \tilde{\nu}(1,w'), >0} \ar[r] & \overline{{}^{I^\flat}\!\! {\tilde{\CB}}_{\tilde{\nu}(r,r), \tilde{\nu}(1,w'), >0}} & {}^{I^\flat}\!\! {\tilde{\CB}}_{\tilde{\nu}(r,r), \tilde{\nu}(1,w'), >0} \bigcap {}^{I^\flat}\!\! \mathring{\tilde{\CB}}_{\tilde{\nu}(r,r), \tilde{\nu}(v,w)}. \ar[l]
}
\]

The statement for $(1, w)$ with sufficiently large $w$ is proved by direct computation. Note that for any $w \in W$, $U^-_{w, >0} P^+_K=U^-_{w^K, >0} P^+_K \subset \dot r U^- P^+_K$ by Lemma \ref{lem-in-u}. For computational purpose, we can further relax the condition and consider the image of $U^-_{w, >0} \subset \dot r U^- P^+_K$ under the map $f_r$ for sufficiently large elements in $W$ instead in $W^K$. Let $w \in W$ be such that $\ell(r^{-1}w) = \ell(w) - \ell(r)$. We write $u = r^{-1} w$. The rest of the proof consists of direct computation of the map $f_r$.

Let $h \in U^-_{w, >0}$. By \eqref{eq:GKL}, we have 
\begin{equation}\label{eq:star1}
\dot{r}^{-1} h \in (U^- \bigcap \dot{r}^{-1} U^+ \dot{r}) h_1 b_1 \quad \text{for some } h_1 \in U^{-}_{u^{-1}, >0} \text{ and } b_1 \in B^+_{\ge 0}.
\end{equation}
We also have 
\begin{equation}\label{eq:star2}
h_1 b_1=b_2 h_2, \quad  \text{for some }h_2 \in U^{-}_{u^{-1}, >0}, b_2 \in B^+_{\ge 0}.
\end{equation}
Set $g=h b_1 \i \pi_K(h_1) \i \dot r \i$. Since $r \in W^K$,  we have $U^- \bigcap \dot{r}^{-1} U^+ \dot{r} \subset U_{P^-_K}$ and thus $g \in \dot r U_{P^-_K} \dot r \i$. We  have $(g \dot r) \i \in \pi_K(h_1) b_1 h \i U^+=\pi_K(h_1) h_1 \i \dot r \i U^+$. 

By Theorem \ref{thm:CB} (3), $\s_{r, +}(g) \dot r \in h_3 B^+$ for some $h_3 \in U^-_{r, >0}$.  Since $h_1 b_1=b_2 h_2$, we have $\pi_J(h_1) b_1=b_2 \pi_J(h_2)$. Then $\s_{r, +}(g) \dot r \in U^- g \dot r=U^- h b_1 \i \pi_K(h_1) \i=U^- b_2 \i$. Thus $\s_{r, +}(g) \dot r=h_3 b_2 \i$. We have 
 
\begin{equation}\label{eq:star3}
h_1 \pi_K(b_1) = b_4 h_4 \quad  \text{for some }h_4 \in U^-_{u \i, >0} \text{ and } b_4 \in B_{K, \ge 0}^+. 
\end{equation}
Then $\pi_K(h_1) \pi_K(b_1)=\pi_K(b_2) \pi_K(h_2) = \pi_K(b_4) \pi_K(h_4)$. So $\pi_K(h_2)=\pi_K(h_4)$ and $$\pi_K(b_2) \i \pi_K(h_1) h_1 \i=\pi_K(h_4) \pi_K(b_1) \i h_1 \i \in \pi_K(h_4) h_4 \i b_4^{-1}.$$ Now we have 
\begin{align}
\notag 	 (\s_{r, +}(g) \dot r )^\flat (\dot r \i \s_{r, -}(g) \i)^\sharp \tilde{U}^+ =& (h_3 b_2 \i)^\flat (\pi_K(h_1) h_1 \i  \dot r \i )^\sharp \tilde{U}^+ \\
\notag	 = & h_3^\flat (\pi_K(b_2) \i \pi_K(h_1) h_1 \i \dot r \i)^\sharp \tilde{U}^+ \\
\label{eq:star4}	 = & h_3^\flat ( \pi_K(h_4) h_4 \i b_4^{-1} \dot r \i)^\sharp \tilde{U}^+.
\end{align}

Hence $\tilde c_r(h \cdot P_K^+/P_K^+) =h_3^\flat (\pi_K(h_4) h_4 \i \dot r \i)^\sharp \tilde B^+ \in {}^{I^\flat}G_{\tilde{\nu}(r,r), \tilde{\nu}(1,w), >0} \cdot \tilde{B}^+$.

By Proposition~\ref{prop:closure-p} (3), $\CP_{K, (v, w), >0}$ is a connected component of $\mathring{\CP}_{K, (v, w)}(\BR) \bigcap \dot r U^-  P^+_K/P^+_K$. By Proposition~\ref{prop:typeII}, ${}^{I^\flat} G_{\tilde{\nu}(r,r), \tilde{\nu}(1,w), >0} \cdot \tilde{B}^+/\tilde{B}^+ =  {}^{I^\flat}\!\! \mathring{\tilde{\CB}}_{\tilde{\nu}(r,r), \tilde{\nu}(1,w), >0}$ is a connected component of ${}^{I^\flat}\!\! \mathring{\tilde{\CB}}_{\tilde{\nu}(r,r), \tilde{\nu}(v,w)}(\BR)$. 

Since the isomorphism $\tilde{c}_r$ sends the connected components of $\mathring{\CP}_{K, (v ,w)}(\BR) \bigcap \dot r U^-(\BR) P^+_K/P^+_K$ to the connected components of ${}^{I^\flat}\!\! \mathring{\tilde{\CB}}_{\tilde{\nu}(r,r), \tilde{\nu}(v,w)}(\BR)$, we have $\tilde{c}_r(\CP_{K, (v ,w), >0})={}^{I^\flat}\!\! {\tilde{\CB}}_{\tilde{\nu}(r,r), \tilde{\nu}(v,w), >0}$. 
\end{proof}

\begin{prop} \label{prop:compatible2}
Let $v \leJ w$ and $x \in W$. Then the isomorphism $i^{\spade}_x: \OJCB_{v, w} \cong {}^{{I^!}^\flat} \!\!\OCB^\spade_{v^\sharp (s_0 x)^\sharp, w^\sharp (s_0 x)^\sharp}$ restricts to an isomorphism $$i^{\spade}_x: \JCB_{v, w, >0}  \cong {}^{{I^!}^\flat} \!\!\CB^\spade_{v^\sharp (s_0 x)^\sharp, w^\sharp (s_0 x)^\sharp, >0}.$$
\end{prop}
 
\begin{proof} 
We first consider the case where $w \in {}^J W$. Fix a reduced expression ${\bf w}$ of $w$. 
Fix a reduced expression of ${\bf s_0 x }$ of $s_0 x$. Then ${\bf w^\sharp s^\sharp_0 x^\sharp}$ is a reduced expression of $w^\sharp (s_0 x)^\sharp$. In particular, we have 
\[
({}^J G_{v, {\bf w}, >0})^\sharp \dot{s}^\sharp_0 \dot{x}^\sharp= {}^{{I^!}^\flat}\!\!G_{v^\sharp (s_0 x)^\sharp, {\bf w^\sharp s^\sharp_0 x^\sharp}, >0}.
\]
Now the statement follows from Proposition~\ref{prop:typeII} (3).

We then consider the general case. Set $v'=v^\sharp (s_0 x)^\sharp, w'=w^\sharp (s_0 x)^\sharp$ and ${}^J w'=({}^J w)^\sharp (s_0 x)^\sharp$. By Proposition \ref{prop:typeII} (2), we have $\overline{\JCB_{v, {}^J w, >0}} \bigcap \OJCB_{v, w}=\JCB_{v, w, >0}$ and $\overline{{}^{{I^!}^\flat}\!\! \CB^\spade_{v', {}^J w', >0}} \bigcap {}^{{I^!}^\flat} \!\!\OCB^\spade_{v', w'}={}^{{I^!}^\flat} \!\!\CB^\spade_{v', w', >0}$. Now the statement follows from the following commutative diagram:
\[
\xymatrix{
\overline{\JCB_{v, {}^J w, >0}} \ar[d]^-{\cong} & \overline{\JCB_{v, {}^J w, >0}} \bigcap \OJCB_{v, w} \ar[d]^-{\cong} \ar[r] \ar[l] & \OJCB_{v, w} \ar[d]^-{\cong} \\
\overline{{}^{{I^!}^\flat}\!\! \CB^\spade_{v', {}^J w', >0}} \bigcap \text{Im}(i^{\spade}_x) & \overline{{}^{{I^!}^\flat}\!\!\CB^\spade_{v', {}^J w', >0}} \bigcap {}^{{I^!}^\flat}\!\! \OCB^\spade_{v', w'} \ar[r] \ar[l] & {}^{{I^!}^\flat}\!\!\OCB^\spade_{v', w'}.
}
\]
\end{proof}


\subsection{Matrix coefficients and admissible functions}\label{sec:delta} 

\subsubsection{Admissible functions}\label{sec:adm} Recall \cite[\S 1.2]{Lu-2} that a function $f: \BR_{> 0}^{m} \times \BR^{n}_{> 0} \rightarrow \BR_{\ge 0}$ is called admissible if $f = f'/f''$, where $f', f'' \in \BZ_{\ge 0}[t_1, \dots, t_m, t'_{1}, \dots, t'_n]$ with $f'' \neq 0$. Note that 

(a) the value of an admissible function is either always $0$ or never $0$.

The following result proved in \cite[Lemma~5.9]{GKL} is useful to prove admissibility. 

(b) Suppose that $f: \BR_{>0}^{m} \times \BR^{n}_{\ge 0} \rightarrow \BR_{\ge0}$ is continuous and the restriction $f: \BR_{>0}^{m} \times \BR^{n}_{>0} \rightarrow \BR_{\ge 0}$ is an admissible function.  Then the restriction $f: \BR_{>0}^{m} \rightarrow \BR_{\ge0}$ is also admissible. 

Let $v \le w$ and ${\bf w}$ be a reduced expression of $w$. Then we have an isomorphism $\BR_{>0}^{\ell(w) - \ell(v)} \rightarrow G_{{\bf v_+}, {\bf w}, >0} $. Let $v' \le w'$ and ${\bf w'}$ be a reduced expression of $w'$. We call  a map $G_{{\bf v_+}, {\bf w}, >0} \rightarrow  G_{{\bf v'_+}, {\bf w'}, >0} $ admissible if the composition $\BR_{>0}^{\ell(w) - \ell(v)} \rightarrow G_{{\bf v_+}, {\bf w}, >0}  \rightarrow G_{{\bf v'_+}, {\bf w'}, >0} \rightarrow \BR_{>0}^{\ell(w') - \ell(v')} \to \Rp$ is admissible, where the last map is the projection to any coordinate. Thanks to \cite[\S 6.1]{BH20}, admissibility is independent of the choice of reduced expressions. We define the admissibility for  maps $G_{{\bf v_+}, {\bf w}, >0} \to T_{>0}$ in the same way. 

Let $w \in {}^JW$ and $\mathbf w$ be a reduced expression of $w$. Let $ v \in W$ with ${}^Jv \le w$ and let ${\bf v_J} $ be a reduced expression of $v_J$. A map $ {}^J G_{v, \bf w, >0} \to \BR_{\ge 0}$ is called admissible if its composition with the map $$\b: \BR^{\Jell(w) - \Jell(v)} \to U^-_{{\bf v_J}, >0} \times G_{{}^J v, {\bf w}, >0} \xrightarrow{(\id, \iota)} U^-_{{\bf v_J}, >0} \times G_{{}^J v, {\bf w}, <0} \to {}^J G_{v, \bf w, >0}$$ is admissible.  It follows from the previous discussion that the admissibility of a map  ${}^J G_{v, \bf w, >0} \to \BR_{\ge 0}$ is independent of the reduced expressions of $w$ and $v_J$.

\subsubsection{Matrix coefficients}\label{sec:symmetric}
We keep the notation in \S\ref{sec:link}. For $\l \in X^{++}$, we consider the matrix coefficient  
\[
 \Delta_{\lambda}:   {G} \longrightarrow \BC 
\]
mapping any $g \in {G}$ to the coefficient of $g {\eta}_{\lambda}$ in ${\eta}_{\lambda}$. We denote by $\{0, \ast\} = \BC / \BC^\times$ the set-theoretical quotient. The composition $\Delta_{\lambda}: G \rightarrow \BC \rightarrow \{0, \ast\}$ factors through $\CB$, which we still denote by $\Delta_{\lambda}$. For any $u \in W$, we further define $\Delta_{\lambda, u}: \CB \rightarrow \{0, \ast\} $, $g  {B}^+/{B}^+ \mapsto  \Delta_{\lambda} (\dot{u}^{-1} g  {B}^+/{B}^+)$. Note that the image is independent of the choice of the representative of $u$ in $G$. We then have $g  {B}^+/{B}^+ \in \dot{u}  {U}^- {B}^+ /{B}^+$ if and only if $\Delta_{\lambda, u} (g  {B}^+/{B}^+)=\{\ast\}$. 

Our next goal is to show that $\JCB_{v, w, >0} \subset \dot u U^- B^+/B^+$ for $v \leJ u \leJ w$. Our strategy is as follows. By Corollary \ref{cor:dense}, $\JCB_{v, w, >0}$ is Zariski dense in $\OJCB_{v, w}$. By \S\ref{sec:Jcr} (a), $\OJCB_{v, w} \bigcap \dot u U^- B^+/B^+ \neq \emptyset$. Thus $\JCB_{v, w, >0} \bigcap \dot u U^- B^+/B^+ \neq \emptyset$. In other words, $\ast$ is contained in $\Delta_{\lambda, u} (g  {B}^+/{B}^+)$. We shall then construct an admissible map $\a: \Rp^n \to \BR_{\ge 0}$ so that the following diagram commutes
\[
\xymatrix{
	 	\BR^{\Jell(w)-\Jell(v)}_{>0} \ar[d]^-\a \ar@{->>}[r]^-{\b} & \JCB_{v, w, >0} \ar[d]^-{\D_{\l, u}} \\
		\BR_{\ge 0} \ar[r] & \{0, \ast\}.
}
\]
By \S\ref{sec:adm} (a), $\D_{\l, u}: \JCB_{v, w, >0} \to \{0, \ast\}$ is constant. So $\Delta_{\lambda, u} (g  {B}^+/{B}^+)=\{\ast\}$ and $\JCB_{v, w, >0} \subset \dot u U^- B^+/B^+$. 

\subsection{Some admissible functions}
In this subsection, we consider some admissible functions arising from the group $G$.
\begin{lem}\label{lem:w-v}
Let $v \le w$. Define $f_{v, +}: U^-_{w, >0} \to U^-_{v, >0}$ by $c_{v, +}(g B^+/B^+)=f_{v, +}(g) B^+/B^+$ for any $g \in U^-_{w, >0}$. Then $f_{v, +}$ is admissible. 

\end{lem}
\begin{proof}
Let $h \in U^-_{w, >0}$. By \eqref{eq:GKL}, we have $\dot{v}^{-1} h \in U^- gt$ for some $g \in U^+_{v^{-1}, > 0}$ and $t \in T_{>0}$. By \eqref{eq:GKL} again, we have $\dot{v} g = g_1 h_1 t_1$ for some $g_1 \in U^+ \bigcap  \dot{v} U^-  \dot{v}^{-1}, h_1 \in U^-_{v, >0}$ and $t_1 \in T_{>0}$. Hence $h  t^{-1}  t^{-1}_1 h_1^{-1} g_1^{-1} =h t^{-1} g^{-1} \dot{v}^{-1} \in   \dot{v} U^-  \dot{v}^{-1}$. So $t = t_1^{-1}$ and $ g_1^{-1} = \sigma_{v,+}(h  t^{-1} g^{-1} \dot{v}^{-1})$. Hence we conclude that 
\begin{align*}
	 c_{v,+}(h B^+/B^+) &= \sigma_{v,+}(h  t^{-1} g^{-1} \dot{v}^{-1}) \dot{v} B^+/B^+ = g_1^{-1}   \dot{v} B^+/B^+ \\
	 & = h_1 t_1g^{-1}   B^+/B^+ =  h_1 B^+/B^+.
\end{align*}
Thus $f_{v, +}(h)=h_1$. It is clear from the construction that $h \mapsto g \mapsto h_1$ is admissible. 
\end{proof}

\begin{lem}\label{lem:ad}
Let $w_1, w_2 \in {}^J W$ and $v_1, v_2 \in W_J$ with $w_1 \le w_2$ and $v_1 \le v_2$. Fix a reduced expression ${\bf w_2}$ of $w_2$. Then ${}^J G_{v_2 w_1, \bf w_2, >0} \subset \dot w_1 \i \dot v_1 \i U^- B^+/B^+$ and we have the following commutative diagram
\[
\xymatrix{
U^-_{v_2, >0} \times G_{w_1, w_2, <0} \ar[r]^-{\cong} \ar[d]_-{(f_{v_1, +}, \text{id})} & {}^J G_{v_2 w_1, \bf w_2, >0} \ar[r]^-{\cong} & \JCB_{v_2 w_1, w_2, >0} \ar[d]^-{{}^J c_{v_1 w_1, -}} \\
U^-_{v_1, >0} \times G_{w_1, w_2, <0} \ar[r]^-{\cong} & {}^J G_{v_1 w_1, \bf w_2, >0} \ar[r]^-{\cong} & \JCB_{v_1 w_1, w_2, >0}.
}
\]
\end{lem}

\begin{proof}
Let $g \in U^-_{v_2, >0}$ and $h \in G_{w_1, w_2, <0}$. By Lemma \ref{lem:w-v}, $f_{v_1, +}(g) \in U^-_{v_1, >0}$.  By definition, $f_{v_1, +}(g) \in (\dot v_1 U_J^- \dot v_1 \i \bigcap U_J^-) g$. Set $p=g \pi_J(h \dot w_1 \i) \i h B^+/B^+ \in \JCB_{v_2 w_1, \bf w_2, >0}$ and $p'=c_{v_1, +}(g) \pi_J(h \dot w_1 \i) \i h B^+/B^+ \in \JCB_{v_1 w_1, \bf w_2, >0}$. Then 
\begin{align*} p & \in (\dot v_1 U_J^- \dot v_1 \i \bigcap U_J^-) p' \subset (\dot v_1 U_J^- \dot v_1 \i) \OJCB_{v_1 w_1, w_2} \\ & \subset (\dot v_1 U_J^- \dot v_1 \i) \, \JB^- \dot v_1 \dot w_1 B^+/B^+ \subset (\dot v_1 U_J^- \dot v_1 \i) \dot v_1 \dot w_1 U^- B^+/B^+ \\ & \subset \dot v_1 U^-_J \dot w_1 U^- B^+/B^+ \subset \dot v_1 \dot w_1 U^- B^+/B^+.
\end{align*}

Hence ${}^J c_{v_1 w_1, -}(p)$ is defined. 

By definition, ${}^J c_{v_1 w_1, -}(p)$ is the unique element in $\OJCB_{v_1 w_1, w_2} \bigcap (\dot v_1 \dot w_1 U^- (\dot v_1 \dot w_1) \i \bigcap \JB^+) p$. Note that $\dot v_1 U_J^- \dot v_1 \i \subset \dot v_1 \dot w_1 U^- (\dot v_1 \dot w_1) \i$. Thus $\dot v_1 U_J^- \dot v_1 \i \bigcap U_J^- \subset \dot v_1 \dot w_1 U^- (\dot v_1 \dot w_1) \i \bigcap \JB^+$ and ${}^J c_{v_1 w_1, -}(p)=p'$. 
\end{proof}

\subsection{Further admissible functions related with $\tilde{G}$}
In this section, we consider admissible functions arising from the group $\tilde{G}$ that are related with the morphism $\tilde{c}_r$ in \S\ref{sec:BBatlas}.

For $\tilde{w} \in \tilde{W}$, we define
$$\tilde{f}_{\tilde{w}}: \dot{\tilde{w}} \tilde{U}^{-} \tilde{B}^+  \rightarrow \tilde{G}, \quad \tilde{g}  \dot{\tilde{w}} \tilde{b} \mapsto {}^{I^{\flat}}\!\!\sigma_{\dot{w}, -} (\tilde{g}) \text{ for } \tilde{g} \in  \dot{\tilde{w}} \tilde{U}^{-}  \dot{\tilde{w}}^{-1} \text{ and } \tilde b \in \tilde B^+.
$$
For $(s,r) \in Q_K$, we define the map
\[
	 f_{(s,r)} = \tilde{f}_{\tilde{\nu}(s,r)} \circ  f_r: \Big(\dot r U^- P^+_K \Big) \bigcap f_r ^{-1} \Big(\dot{\tilde{\nu}}(s,r)   \tilde{U}^{-} \tilde{B}^+ \Big) \rightarrow \tilde{G}.
\]
Here the map $f_r$ is defined in \S\ref{sec:BBatlas}. Note that $f_{(r,r)}  = f_r$, since $f_r \Big(\dot r U^- P^+_K  \Big) \subset \dot{\tilde{\nu}}(r,r)  \tilde{U}^{-} \tilde{B}^+$.

\begin{lem}\label{lem:1}
Let $r \in W^K$ and $w \in W$ with $\ell(r \i w)=\ell(w)-\ell(r)$. Then the map $$U^-_{u, >0} \times U^-_{w, >0} \to \BR, \qquad (g, h) \mapsto   \Delta_{\lambda}(\dot{t}^{\sharp} \dot{u}^{\flat, -1} g f_{(s,r)}(h))$$ is admissible for any $(u,t) \in Q_K$.
\end{lem}  

\begin{remark}
By Proposition~\ref{prop:closure-p} and Lemma~\ref{lem:ad}, $U^-_{w, >0} \subset (\dot r U^- P^+_K) \bigcap f_r ^{-1} (\dot{\tilde{\nu}}(s,r)   \tilde{U}^{-} \tilde{B}^+)$. So the map $f_{(s, r)}$ is defined on $U^-_{w, >0}$. 
\end{remark}

\begin{proof}
By the proof of Proposition~\ref{prop:compatible} (in particular \eqref{eq:star4}), we have 
$$f_r(h) \in h_3^\flat ( \pi_K(h_4) h_4 \i b_4 \i  \dot r \i)^\sharp \tilde U^+=h_3^\flat ( \pi_K(h_4) h_4 \i z_4  \dot r \i)^\sharp \tilde U^+$$ for some $h_3 \in U^-_{r, >0}$, $h_4 \in U^-_{r \i w, >0}$ and $b_4 \in B^+_{K, \ge 0}$. Here $z_4 \in T_{>0}$ with $b_4 \i \in z_4 U^+_K$. Moreover, by \eqref{eq:star1}--\eqref{eq:star4} in the proof of Proposition~\ref{prop:compatible}, all maps $h \mapsto h_3$, $h \mapsto h_4$, $h \mapsto z_4$ are admissible. By Lemma~\ref{lem:ad}, $f_{(s, r)}(h) \in (c_{s,+}(h_3))^\flat ( \pi_K(h_4) h_4 \i z_4  \dot r \i)^\sharp \tilde U^+$. By Lemma~\ref{lem:w-v}, the map $h_3 \mapsto f_{s,+}(h_3)$ is admissible. By \eqref{eq:GKL}, we have $ \dot{u}^{-1} (g\, \,f_{s,+}(h_3) ) \in U^{-} c$ for some  $c_1 \in B_{\ge 0}$. By \S\ref{sec:GKL} (a), we have $c^\flat( \pi_K(h_4) h_4 \i z_4  \dot r \i)^\sharp  \tilde{U}^+= ( \pi_K(h_5) h_5 \i z_5  \dot r \i)^\sharp  \tilde{U}^+$ for some $h_5  \in U^-_{u^{-1},>0} \text{ and } z_5 \in  {T}_{>0}$.

Moreover, the maps $(g, f_{s,+}(h_3) ) \mapsto c$, $(c, h_4) \mapsto h_5$ and $(c, h_4, z_4) \to z_5$ are all admissible. Since $t \in W^K$, we have $t^\sharp \in \tilde W^{I^\flat}$. Thus 
$$
\dot{t}^{\sharp} \dot{u}^{\flat, -1} \,\, g^\flat (c_{s,+}(h_3) )^\flat ( \pi_K(h_4) h_4 \i z_4  \dot r \i)^\sharp \tilde U^+ \in \tilde{U}^- \dot{t}^\sharp  h_5^{\sharp, -1} z_5^\sharp \dot{r}^{-1, \sharp}  \tilde{U}^+.
$$
So \begin{align*}
\Delta_{\lambda}(\dot{t}^{\sharp} \dot{u}^{\flat, -1} g f_{(s,r)}(h)) & = \Delta_{\lambda}(\dot{t}^\sharp   h_5^{\sharp, -1} z_5^\sharp \dot{r}^{-1, \sharp})= \Delta_{\lambda}( \iota(\dot{t}^\sharp   h_5^{\sharp, -1} z_5^\sharp \dot{r}^{-1, \sharp})) \\ &= \Delta_{\lambda}( \iota(\dot{t}^\sharp   )  \iota(h_5^{\sharp, -1}) z_5^\sharp \iota(\dot{r}^{-1, \sharp})).
\end{align*} 
Note that $h_5 \mapsto \iota(h_5^{-1})$ is admissible, $\iota(\dot{t}) = \dot{x}^{-1}$ for $ x = t^{-1} \in W$ and  $\iota(\dot{r}^{-1}) = \dot{y}$ for $ y = r^{-1}\in W$. By \cite[Proposition 5.13]{GKL}, the map $(h_5, z_5) \mapsto \Delta_{\lambda}( \iota(\dot{t}^\sharp   )  \iota(h_5^{\sharp, -1}) z_5^\sharp \iota(\dot{r}^{-1, \sharp}))$ is admissible. The lemmas follows now. 
\end{proof}
\begin{lem}\label{lem:2}
Let $(s, r) \in Q_K$ and $w \in W$ with $r \le w$. Then the map $$U^-_{w, >0} \to \BR, \qquad h \mapsto  \Delta_{\lambda}(\dot{t}^{\sharp} \dot{u}^{\flat,-1} f_{(s,r)}(h))$$ is admissible  for any $(u,t) \in Q_K$.
\end{lem}

\begin{proof}
We simply write $\D$ for $  \Delta_{\lambda}(\dot{t}^{\sharp} \dot{u}^{\flat,-1} \cdot -)$. Let $w_1 \in W$ with $\ell(r^{-1} w_1) = \ell(w_1) - \ell(r)$ and $w \le w_1$. We fix reduced expressions of $u$ and $w_1$ (and thus the positive subexpression for $w$). The statement is proved using the following commutative diagram. 
\[
\xymatrix{
\BR_{>0}^{\ell(u)} \times \BR_{>0}^{\ell(w_1)} \ar[r] \ar[d] & \BR_{\ge 0}^{\ell(u)} \times (\BR_{\ge 0}^{\ell(w_1)-\ell(w)} \times \BR_{>0}^{\ell(w)}) \ar[d] & \BR_{>0}^{\ell(w)} \ar[d] \ar[l] \\
U^-_{u, >0} \times U^-_{w_1, >0} \ar[r] \ar[dr]_-{\Delta \circ m \circ (id, f_{(s,r)})} & (\bigsqcup_{1 \le u' \le u} U^-_{u', >0}) \times (\bigsqcup_{w \le w' \le w_1} U^-_{w', >0})  \ar[d] & U^-_{w, >0} \ar[l] \ar[ld]^-{\D \circ f_{(s,r)}} \\
& \BR &
}
\]

Now let us explain how the maps are defined. Let $\mathbf u=s_{i_1} \cdots s_{i_n}$ be the reduced expression we fixed in the beginning. The map $\BR_{\ge 0}^{\ell(u)} \to U^-$ is defined by $(a_1, \ldots, a_n) \mapsto y_{i_1}(a_1) \cdots y_{i_n}(a_n)$. It is easy to see that the image is $\bigsqcup_{1 \le u' \le u} U^-_{u', >0}$. By a similar argument to \cite[Proposition~4.2]{Lus-1}, $\bigsqcup_{1 \le u' \le u} U^-_{u', >0}$ is a closed subspace of $G$. Similarly, we have a continuous map $\BR_{\ge 0}^{\ell(w_1)} \to \bigsqcup_{1 \le w' \le w_1} U^-_{w', >0}$. If we further require that all the coordinate associated to the positive subexpression of $w$ are positive, then we obtain a continuous map $\BR_{\ge 0}^{\ell(w_1)-\ell(w)} \times \BR_{>0}^{\ell(w)} \to G$ and the image of this map is the locally closed subspace $\bigsqcup_{w \le w' \le w_1} U^-_{w', >0}$ of $G$. 

Note that for any $w' \in W$ with $w \le w'$, we have $r \le w'$. Thus $U^-_{w', >0} \subset \dot r U^- P^+_K$. Therefore we have a continuous map $$\BR_{\ge 0}^{\ell(u)} \times (\BR_{\ge 0}^{\ell(w_1)-\ell(w)} \times \BR_{>0}^{\ell(w)}) \to (\bigsqcup_{1 \le u' \le u} U^-_{u', >0}) \times (\bigsqcup_{w \le w' \le w_1} U^-_{w', >0}) \xrightarrow{\D \circ m \circ (id, f_{(s,r)})} \BR.$$

By Lemma \ref{lem:1}, the restriction to $\BR_{>0}^{\ell(u)} \times \BR_{>0}^{\ell(w_1)}$ is admissible. Hence by \S\ref{sec:adm} (b), the map $\D \circ f_{(s,r)}: U^-_{w, >0} \to \BR$ is admissible. 
\end{proof}

\begin{lem}\label{lem:3}
Let $r \in W^K$ and $v, w \in W$ with $v \le r \le w$. We fix a reduced expression ${\mathbf w}$ of $w$.  Then the map $$G_{{\bf v_+},{\bf w}, >0} \to \BR, \qquad h \mapsto   \Delta_{\lambda}(\dot{t}^{\sharp} \dot{u}^{\flat,-1} f_{(r,r)}(h))$$ is admissible  for any $(u,t) \in Q_K$.
\end{lem}

\begin{remark}
By Proposition~\ref{prop:closure-p} that $G_{{\bf v_+},{\bf w}, >0}  \subset   \dot r U^- P^+_K$. So $f_{(r, r)}(G_{{\bf v_+},{\bf w}, >0})$ is defined. However, we have not proved
$G_{{\bf v_+},{\bf w}, >0}  \subset  \Big(\dot r U^- P^+_K \Big) \bigcap f_r ^{-1} \Big(\dot{\tilde{\nu}}(s,r) \tilde{U}^{-} \tilde{B}^+ \Big)$ yet and thus we can not apply the general map $f_{(s, r)}$ to $G_{{\bf v_+},{\bf w}, >0}$. The general case will be handled after Lemma \ref{lem:3} and Corollary \ref{cor:rrinU} are established.  
\end{remark}

\begin{proof} We simply write $\D$ for $\Delta_{\lambda}(\dot{t}^{\sharp} \dot{u}^{\flat,-1} \cdot -)$. By \cite[Proposition~6.2]{BH20}, $U^+_{v \i, >0} \CB_{v, w, >0}=\CB_{1, w, >0}$ and the induced map $U^+_{v \i, >0} \times G_{{\bf v_+},{\bf w}, >0} \to U^-_{w, >0}$ is admissible. By Lemma \ref{lem:2}, the map $$U^+_{v \i, >0} \times G_{{\bf v_+},{\bf w}, >0} \to U^-_{w, >0} \xrightarrow{\D} \BR$$ is admissible. Moreover, for any $v' \le v$, $U^+_{(v') \i, >0} \CB_{v, w, >0}=\CB_{(v') \i \circ_l v, w, >0}$ and $(v') \i \circ_l v \le v \le r$. Thus $U^+_{(v') \i, >0} G_{{\bf v_+},{\bf w}, >0} \subset \dot r U^- P^+_K$. Therefore the continuous map $\BR_{\ge 0}^{\ell(u)} \times G_{{\bf v_+},{\bf w}, >0} \to (\bigsqcup_{1 \le v' \le v} U^-_{(v') \i, >0}) \times G_{{\bf v_+},{\bf w}, >0} \xrightarrow{m} G$ has image inside $\dot r U^- P^+_K$. Thus we have a continuous map $\BR_{\ge 0}^{\ell(u)} \times G_{{\bf v_+},{\bf w}, >0} \to \BR$.

Now the statement follows from the following commutative diagram using the similar proof of Lemma \ref{lem:2}:

\[
\xymatrix{
\BR_{>0}^{\ell(v)} \times \BR_{>0}^{\ell(w)-\ell(v)} \ar[r] \ar[d] & \BR_{\ge 0}^{\ell(v)} \times \BR_{>0}^{\ell(w)-\ell(v)} \ar[d] & \BR_{>0}^{\ell(w)-\ell(v)} \ar[d] \ar[l] \\
U^+_{v \i, >0} \times G_{{\bf v_+},{\bf w}, >0} \ar[r] \ar[dr]_-{\Delta \circ  f_{(r,r)} \circ m} & (\bigsqcup_{1 \le v' \le v} U^+_{(v') \i, >0}) \times G_{{\bf v_+},{\bf w}, >0}  \ar[d] & G_{{\bf v_+},{\bf w}, >0} \ar[l] \ar[ld]^-{\D \circ f_{(r,r)}} \\
& \BR &
.}
\]
\end{proof}

\begin{cor}\label{cor:rrinU}
Let $(r,r) \le (s,r) \le (v,w) \in Q_K$. We have 
\begin{enumerate} 
\item
${}^{I^\flat}\!\! \mathring{\tilde{\CB}}_{\tilde{\nu}(r,r), \tilde{\nu}(v,w), >0} \subset \dot{\tilde{\nu}}(s, r)  \tilde U^- \tilde B^+/\tilde B^+$;
\item 
$G_{{\bf v_+},{\bf w}, >0}  \subset  \Big(\dot r U^- P^+_K \Big) \bigcap f_r ^{-1} \Big(\dot{\tilde{\nu}}(s,r)   \tilde{U}^{-} \tilde{B}^+ \Big)$.
\end{enumerate}
\end{cor}

\begin{proof}We have a commutative diagram 
\[
\xymatrix{
G_{{\bf v_+},{\bf w}, >0} \ar[r] \ar[d] & \CP_{K, (v, w), >0} \ar[r] &
{}^{I^\flat}\!\! \mathring{\tilde{\CB}}_{\tilde{\nu}(r,r), \tilde{\nu}(v,w), >0} \ar[d]^{\Delta_{\lambda,   {\tilde{\nu}}(s,r)}} \\
\BR \ar[rr] & & \{0, \ast\}.
}
\]

By Lemma \ref{lem:3}, the map $G_{{\bf v_+},{\bf w}, >0} \to \BR$, $h \mapsto   \Delta_{\lambda}(\dot{r}^{\sharp} \dot{s}^{\flat,-1} f_{(r,r)}(h))$ is admissible. Hence the composition $G_{{\bf v_+},{\bf w}, >0} \to  \{0, \ast \}$ is constant. By Proposition~\ref{prop:compatible}, all maps in the first row of the commutative diagram are surjective. Then the map $\Delta_{\lambda,   {\tilde{\nu}}(s,r)}: {}^{I^\flat}\!\! \mathring{\tilde{\CB}}_{\tilde{\nu}(r,r), \tilde{\nu}(v,w), >0} \to \{0, \ast \}$ is constant. Hence by \S\ref{sec:delta} (a), we have 
${}^{I^\flat}\!\! \mathring{\tilde{\CB}}_{\tilde{\nu}(r,r), \tilde{\nu}(v,w), >0} \subset \dot \nu(s, r) \tilde U^- \tilde B^+/\tilde B^+$. Now the statements follow from \S\ref{sec:symmetric}. 
\end{proof}

The following lemma can be proved entirely similar to Lemma~\ref{lem:3}, thanks to Corollary~\ref{cor:rrinU}.
\begin{lem}\label{lem:4}
Let $(r,r) \preceq (s,r) \preceq (v,w) \in Q_K$. Fix a reduced expression ${\mathbf w}$ of $w$. Then the map $$G_{{\bf v_+},{\bf w}, >0} \to \BR, \qquad h \mapsto  \Delta_{\lambda}(\dot{t}^{\sharp} \dot{u}^{\flat,-1} f_{(s,r)}(h))$$ is admissible  for any $(u,t) \in Q_K$.
\end{lem}

\subsection{Proof of Proposition \ref{prop:J}} \label{sec:pf1}A special case of part (1) has already been proved in Corollary~\ref{cor:II}. The general case follows by Theorem~\ref{thm:product} and Corollary~\ref{cor:product} once we have verified the assumptions in Theorem~\ref{thm:product}, that is, Proposition \ref{prop:J} part (2) \& (3).

By Proposition \ref{prop:compatible2}, it suffices to prove part (2) \& (3) for basic $J$-Richardson varieties. Namely, 

(a) for $(s, r) \preceq (v, w)$ in $Q_K$, ${}^{I^\flat} \! \! \CB_{\nu(s, r), \nu(v, w), >0}$ is a connected component of ${}^{I^\flat}  \! \! \OCB_{\nu(s, r), \nu(v, w)}(\BR)$;

(b) for $(s, r) \preceq (u,t) \preceq (v, w)$ in $Q_K$, ${}^{I^\flat} \! \! \CB_{\nu(s, r), \nu(v, w), >0} \subset  \dot{\tilde{\nu}}(u,t) \tilde U^- \tilde B^+/\tilde B^+$.

We first show (a). The case where $s=r$ follows from Proposition \ref{prop:closure-p} (3) and Proposition \ref{prop:compatible}. The rest follows from Lemma \ref{lem:key2} (3) thanks to Corollary~\ref{cor:rrinU}. 
We now show part (b). We have a commutative diagram 
 \[
\xymatrix{
G_{{\bf v_+},{\bf w}, >0} \ar[r] \ar[d] & \CP_{K, (v, w), >0} \ar[r] &
{}^{I^\flat}\!\! \mathring{\tilde{\CB}}_{\tilde{\nu}(r,r), \tilde{\nu}(v,w), >0} \ar[r] & {}^{I^\flat}\!\! \mathring{\tilde{\CB}}_{\tilde{\nu}(s,r), \tilde{\nu}(v,w), >0} \ar[d]^{\Delta_{\lambda,   {\tilde{\nu}}(u,t)}} \\
\BR \ar[rrr] & & & \{0, \ast \}.
}
\]

By Lemma~\ref{lem:4}, the map $G_{{\bf v_+},{\bf w}, >0} \rightarrow \BR$, $h \mapsto \Delta_{\lambda}(\dot{t}^{\sharp} \dot{u}^{\flat,-1} f_{(s,r)}(h))$, is admissible. So the map $G_{{\bf v_+},{\bf w}, >0} \rightarrow \{0, \ast \}$ is constant. 

By \eqref{eq:5.1} and Proposition \ref{prop:closure-p} (1), the map $G_{{\bf v_+},{\bf w}, >0} \to \CP_{K, (v, w), >0}$ is surjective. By Proposition \ref{prop:compatible}, the map $\CP_{K, (v, w), >0} \to
{}^{I^\flat}\!\! \mathring{\tilde{\CB}}_{\tilde{\nu}(r,r), \tilde{\nu}(v,w), >0}$ is surjective. By part (a) and Lemma \ref{lem:key2}, the map ${}^{I^\flat}\!\! \mathring{\tilde{\CB}}_{\tilde{\nu}(r,r), \tilde{\nu}(v,w), >0} \to {}^{I^\flat}\!\! \mathring{\tilde{\CB}}_{\tilde{\nu}(s,r), \tilde{\nu}(v,w), >0}$ is surjective. Hence the map $\Delta_{\lambda, {\tilde{\nu}}(u,t)}: {}^{I^\flat}\!\! \mathring{\tilde{\CB}}_{\tilde{\nu}(s,r), \tilde{\nu}(v,w), >0} \to \{0, \ast \}$ is constant. By \S\ref{sec:symmetric}, we have 
$
{}^{I^\flat}\!\! \mathring{\tilde{\CB}}_{\tilde{\nu}(s,r), \tilde{\nu}(v,w), >0} \subset  \dot{\tilde{\nu}}(u,t) \tilde U^- \tilde B^+/\tilde B^+$.



\subsection{Proof of Proposition~\ref{prop:CPmanifold}}\label{sec:pf2} 
For any $r \in W^K$ with $(r, r) \preceq (v, w)$, we set $$Y_r=\overline{\CP_{K, (v, w)}, >0} \bigcap \dot r U^- P^+_K/P^+_K.$$ By Proposition \ref{prop:closure-p} (4), $\CP_{K, (v', w'), >0} \subset \dot r U^- P^+_K/P^+_K$ for any $(r,r ) \preceq (v',w') \preceq (v,w)$. Recall that the embedding $\tilde{c}_r$ in \S\ref{sec:BBatlas} is stratified. Then by 
 Proposition \ref{prop:compatible}, the embedding $\tilde{c}_r $ maps $\CP_{K, (v', w'), >0}$ to ${}^{I^\flat}\!\! {\tilde{\CB}}_{\tilde{\nu}(r,r), \tilde{\nu}(v',w'), >0}$. 
 
Note that Theorem~\ref{thm:J} has been fully established now. It follows that $\overline{{}^{I^\flat}\! {\tilde{\CB}}_{\tilde{\nu}(r,r), \tilde{\nu}(v,w), >0}}$ is a topological manifold with boundary 
 \[
 	 \partial \Big( \overline{{}^{I^\flat}\! {\tilde{\CB}}_{\tilde{\nu}(r,r), \tilde{\nu}(v,w), >0}} \Big)= \bigsqcup_{\stackrel{\tilde \nu(r,r) {}^{I^\flat}\!\!\le \tilde w_1 {}^{I^\flat}\!\!\le \tilde w_2 {}^{I^\flat}\!\!\le \tilde \nu (v,w),}{(\tilde w_1, \tilde w_2) \neq (\tilde \nu (r,r), \tilde \nu (v,w))}} {}^{I^\flat}\!\! \mathring{\tilde{\CB}}_{\tilde w_1, \tilde w_2, >0}.
 \] 
 
Hence the boundary of the Hausdorff closure of ${}^{I^\flat}\! {\tilde{\CB}}_{\tilde{\nu}(r,r), \tilde{\nu}(v,w), >0}$ in ${}^{I^\flat}\! \mathring{\tilde{\CB}}^{\tilde{\nu}(r,r)}$ is $$\bigsqcup_{\tilde \nu(r,r) {}^{I^\flat}\!\!\le \tilde w_1 {}^{I^\flat}\!\!< \tilde \nu (v,w)} {}^{I^\flat}\!\! \mathring{\tilde{\CB}}_{\tilde \nu(r, r), \tilde w_1, >0}.$$
 
By \cite[Proposition 4.2 (1) \& (3)]{BH21}, the map $\tilde \nu$ gives a bijection $$(\{(u, t) \in Q_K; (r, r) \preceq (u, t) \preceq (v, w)\}, \preceq) \longleftrightarrow \{\tilde w \in \tilde W; \tilde \nu(r, r)\,  {}^{I^\flat}\!\!\!\!\le \tilde w \, {}^{I^\flat}\!\!\!\!\le \tilde \nu(v, w), {}^{I^\flat}\!\!\!\!\le).$$

   
   Hence $Y_r$ is a topological manifold with boundary $$\partial Y_r=\bigsqcup_{(v', w') \in Q_K; (r, r) \preceq (v', w') \precneq (v, w)} \CP_{K, (v', w'), >0}.$$ 
 
 By Proposition \ref{prop:closure-p} (4), $\overline{\CP_{K, (v, w), >0}}=\bigcup_{r \in W^K; (r, r) \preceq (v, w)} Y_r$ is an open covering. In particular, $\overline{\CP_{K, (v, w), >0}}$ is a topological manifold with boundary $$\bigcup_{r \in W^K; (r, r) \preceq (v, w)} Y_r=\bigsqcup_{(v', w') \in Q_K; (v', w') \preceq (v, w) \text{ and } (r, r) \preceq (v', w') \text{ for some } r \in W^K} \CP_{K, (v', w'), >0}.$$ By definition, $(w', w') \preceq (v', w')$ for any $(v', w') \in Q_K$. In other words, there always exists $r \in W^K$ with $(r, r) \preceq (v', w')$. This finishes the proof. 



\end{document}